\documentclass[11pt]{article}  % default square logo

\RequirePackage{amssymb,amsfonts,amsmath}
\usepackage{amsthm}
\RequirePackage{bm,bbm}
\RequirePackage{epsfig,graphicx}
\RequirePackage{comment}
\usepackage{lscape}
\usepackage{mathtools}
\usepackage{float}
\usepackage{subfigure}
\usepackage[export]{adjustbox}
\usepackage{scalerel}
\usepackage{geometry}
\usepackage{hyperref}
\usepackage{xcolor} 
\hypersetup{
        colorlinks,
    linkcolor=black,
    citecolor={blue!50!black},
    urlcolor={blue!80!black},
    urlcolor={black}
}

\DeclarePairedDelimiter{\ceil}{\lceil}{\rceil}

\usepackage{latexsym}
\usepackage{enumerate}
\usepackage{epsfig}
\usepackage{graphicx}
\usepackage{color}
\usepackage{float}
\usepackage{stmaryrd}
\usepackage{mathrsfs}
\usepackage{makeidx}
\usepackage{fancyhdr}
\usepackage{lastpage}
\usepackage{url}
\usepackage{fullpage}
\usepackage{parskip}
\usepackage{booktabs}
\usepackage{multirow}
\usepackage{comment}
\usepackage{mathtools}

\newtheorem{theorem}{Theorem}

\newtheorem{assumption}{Assumption}

\def \hX  {{\widehat{X}}}
\def \hY  {{\widehat{Y}}}
\def \hx  {{\widehat{x}}}
\def \hy  {{\widehat{y}}}
\def \hz  {{\widehat{z}}}

\def \EE  {{\mathbb{E}}}
\def \PP  {{\mathbb{P}}}
\def \QQ  {{\mathbb{Q}}}
\def \RR  {{\mathbb{R}}}

\def \D   {{\rm d}}

\def \Xo {X^{(0)}}
\def \Xte {X^{(\theta)}}

\def \Rts {R^{(\sigma)}}

\def \Yo {Y^{(0)}}
\def \Yte {Y^{(\theta)}}
\def \Yts {Y^{(\sigma)}}

\def \yo {y^{(0)}}
\def \yte {y^{(\theta)}}

\def \zts {z^{(\sigma)}}

\def \e   {{\rm e}}  % Roman e for exponential

\begin{document}
\pagenumbering{gobble} 
\title{\LARGE {\bf %Mean-Square s
Importance Sampling for Pathwise Sensitivity\\ of Stochastic Chaotic Systems}}
\author{Wei Fang\footnote{Mathematical Institute, University of Oxford, Andrew Wiles Building, Woodstock Road, Oxford, OX2 6GG, UK, E-mail: wei.fang@maths.ox.ac.uk, mike.giles@maths.ox.ac.uk}
\ \,{\small and} \  Michael B. Giles\footnotemark[1]   
 }

\date{}
\maketitle
\setcounter{secnumdepth}{3}
\setcounter{tocdepth}{2}
%\tableofcontents
%\newpage
\pagenumbering{arabic} 

\begin{abstract}
This paper proposes a new pathwise sensitivity estimator for chaotic SDEs. By introducing a spring term between the original and perturbated SDEs, we derive a new estimator by importance sampling. The variance of the new estimator increases only linearly in time $T,$ compared with the exponential increase of the standard pathwise estimator. We compare our estimator with the Malliavin estimator and extend both of them to the Multilevel Monte Carlo method, which further improves the computational efficiency. Finally, we also consider using this estimator for the SDE with small volatility to approximate the sensitivities of the invariant measure of chaotic ODEs. Furthermore, Richardson-Romberg extrapolation on the volatility parameter gives a more accurate and efficient estimator.  Numerical experiments support our analysis.
\end{abstract}


\vspace{1em}


\section{Introduction}
In real-world applications, such as aerodynamic shape optimization, statistics and weather forecast, expectations with respect to the invariant measure of chaotic ordinary differential equations (ODEs) and stochastic differential equations (SDEs) are of great interest. The sensitivities of the quantity with respect to the parameters of the chaotic system is also helpful for calibrations and further considerations. However, the computations of the sensitivities for both ODEs and SDEs are difficult due to the chaotic property. A small perturbation of the coefficient results in a large divergence of the two solutions. Lea, Allen \& Haine \cite{LAH00} use the Lorenz equation (ODE) as an example to illustrate the failure of the pathwise sensitivity method due to the blow-up of the variation process. A similar problem exists for the stochastic Lorenz equation. The variation process may become not ergodic and blow up exponentially, even though the chaotic system itself is ergodic, see section \ref{Numerical results} for detailed numerical results. In this paper, we consider first how to calculate these sensitivities for chaotic SDEs efficiently and then discuss how to use this to approximate the sensitivities of ODEs.

Consider an $m$-dimensional chaotic SDE driven by a $m$-dimensional Brownian motion with parameter $\theta:$
\begin{equation}
\D X_t^{(\theta)} \ = \ f(\theta;X_t^{(\theta)})\, \D t\ + \sigma\,\D W_t,
\label{theta SDE}
\end{equation}
which has a fixed initial data $\xi_0$ and a locally Lipschitz drift $f:\, \mathbb{R}^m\rightarrow \mathbb{R}^m$ such that the strong solution exists uniquely and the SDE is geometric ergodic:
\begin{equation*}
\left|\EE\left[\varphi(X_T^{(\theta)})-\pi^{(\theta)}(\varphi)\right]\right|\leq \mu\,\e^{-\lambda T}
\end{equation*}
for some constant $\mu,\lambda>0,$ where $\varphi$ is a differentiable function with polynomial growth and $\pi(\varphi)$ is the expectation of $\varphi$ with respect to the invariant measure of the SDE. For the computation of $\pi^{(\theta)}(\varphi),$ one common approach is to truncate the infinite time interval into a sufficiently large finite time $T$ and estimate 
\[
F(\theta) \ = \ \EE\left[\varphi(X_T^{(\theta)})\right].
\]
For the computation of the sensitivity of $\pi^{(\theta)}(\varphi)$ with respect to $\theta,$ Assaraf et al. \cite{AJLR18} showed
\[
\frac{\partial\, \pi^{(\theta)}(\varphi)}{\partial \theta} = \frac{\partial}{\partial\theta}\left( \lim_{T\rightarrow \infty} \EE\left[\varphi(X_T^{(\theta)})\right] \right) 
= \lim_{T\rightarrow \infty}  \left( \frac{\partial}{\partial\theta}\EE\left[\varphi(X_T^{(\theta)})\right] \right)
\]
under some restrictive conditions. The exchangeability of the limit and the differentiation for general ergodic SDEs is unknown and in this paper, we assume that this holds and the convergence speed is still exponentially fast:
\begin{equation}
\left| \frac{\partial}{\partial\theta}\EE\left[\varphi(X_T^{(\theta)})\right]  - \frac{\partial\, \pi^{(\theta)}(\varphi)}{\partial \theta}  \right|\leq \mu^*\,\e^{-\lambda^*T}
\label{ConvergenceRateInvariant}
\end{equation}
for some constant $\mu^*,\lambda^*>0.$ Therefore we approximate the sensitivity of the $\pi^{(\theta)}(\varphi)$ by $\partial F/\partial\theta$ with a sufficiently large $T.$

There are several main approaches to computing sensitivities of SDEs \cite{GP13}: finite difference (FD), likelihood ratio method (LR \cite{GP87}), pathwise sensitivity calculation (PS), the weak derivative method (WD \cite{PG88}) and Malliavin calculus (MA \cite{FLLLT99}). The FD estimator is straightforward to implement. For example, we can use the central difference $(F(\theta+\varepsilon/2)-F(\theta-\varepsilon/2))/\varepsilon$ to approximate the first-order derivative. Using the same Brownian paths for both $F(\theta+\varepsilon/2)$ and $F(\theta-\varepsilon/2))$ greatly reduces the variance, but the estimator is biased. Both LR and PS estimators are unbiased. The difference is the choice of the sample space $\Omega$ to do the differentiation \cite{LP90}. By choosing the discrete numerical solution paths as $\omega,$ differentiating their log joint distribution gives the LR estimator. However, the variance of LR estimator in most cases is $O(h^{-1})$ as $h\rightarrow 0$ where $h$ is the timestep size of the numerical approximation. Again, the tradeoff between bias and variance is needed. The PS, by choosing the independent $U(0,1)$ variates as the sample space, is relatively more popular due to its low variance and applicability of adjoint method which improves the efficiency when calculating many sensitivities \cite{GG06}. It requires differentiating the original SDE \eqref{theta SDE} to obtain the variation process. In our case, the plain PS estimator suffers the problem that the variation process blows up exponentially. Similar to LR, the WD method differentiates the probability measures but it may require resimulations which increases the computational cost greatly. Lastly, Malliavin calculus estimators can be viewed as a combination of the LR and PS methods \cite{CG07}, but is less popular in practice due to the heavy machinery of Malliavin calculus. However, we will show later that in our case, the Malliavin calculus provides a comparable estimator.

Another issue for the computation of sensitivities is the extension to function $\varphi$ with discontinuity, which is out of our scope in this paper but may provide possible future research directions. LR, WD and MA all cope with this difficulty but FD and PS suffer. Giles \cite{GB09} proposed the vibrato Monte Carlo method, which can be viewed as a combination of PS and LR for the final timestep. Chan \& Joshi \cite{CJ13} proposed partial proxy schemes by approximately shifting the means of the standard normal variables that govern the system. Tong \& Liu \cite{TL16} derived a new estimator for discontinuous $\varphi$ and used importance sampling to make all the samples fall into the same set to avoid the discontinuity.

In this paper, we first derive the Malliavin estimator for the chaotic SDEs following the idea in \cite{FLLLT99}, since it expresses the sensitivity by the process itself and avoids the variation process. However, it only works for the sensitivity of the drift parameter and fails for others since the estimators again involve the variation process. The other drawback is that it is difficult to apply the adjoint method. Therefore, we propose a new PS estimator with importance sampling. Similar to the idea of introducing a spring term between fine and coarse paths in \cite{Part3}, we add a spring term between $X_t^{(\theta)}$ and $X_t^{(\theta+\varepsilon)}$ and derive the new PS estimator with the Radon-Nikodym derivative and take the differentiation limit $\varepsilon\rightarrow 0.$ The benefit of this change is that the variation process with spring term becomes ergodic again, which allows the computation of sensitivities of the volatility parameter and initial condition.
The variances of both Malliavin and the new PS estimators increase only linearly in time $T$ which is a great improvement compared with the exponential increase of the plain PS estimator. Next, we apply multilevel Monte Carlo method \cite{giles08,giles15} to improve the efficiency further. Since the estimator involves the original ergodic SDE which is not contractive, we need to employ the change of measure technique in \cite{Part3} to add another spring term between the fine and coarse paths for level estimators. Then, we can use the same MLMC scheme with the same random samples and Radon-Nikodym derivatives to calculate the original value $F(\theta)$ and its derivative simultaneously. Finally, we consider the approximation of the sensitivities of chaotic ODEs by the stochastic version with small volatility $\sigma.$ 

The rest of the paper is organised as follows.  Section 2 reviews the pathwise sensitivity method and explains the problem of the plain PS estimator. Section 3 derives the Malliavin estimator and the new pathwise sensitivity method with importance sampling is introduced in Section 4. Numerical results are shown in section 5. The applications of MLMC schemes are presented in section 6. Section 7 has some conclusions and discusses future extensions. 

In this paper we consider the infinite time interval $[0,\infty)$ and let 
$(\Omega,\mathcal{F},\PP)$ be a probability space with 
normal filtration $(\mathcal{F}_t)_{t\in[0,\infty)}$ corresponding to
a $m$-dimensional standard Brownian motion 
$W_t=(W^{(1)},W^{(2)},\ldots,W^{(m)})_t.$
We denote the vector norm by 
$\|v\|\triangleq(|v_1|^2+|v_2|^2+\ldots+|v_m|^2)^{\frac{1}{2}}$, 
the inner product of vectors $v$ and $w$ by 
$\langle v,w \rangle\triangleq v_1w_1+v_2w_2+\ldots+v_mw_m$, 
for any $v,w\in\RR^m$ and the Frobenius matrix norm by 
$\|A\|\triangleq \sqrt{\sum_{i,j}A_{i,j}^2}$ 
for any $A\in\RR^{m\times d}.$

To focus on the computation of the sensitivity, we assume the existence and strong convergence of a suitable numerical scheme.
\begin{assumption}[Numerical Solution]
\label{assp:Numerical solution}
For the ergodic SDE \eqref{theta SDE}, we assume that there exists a numerical solution $\hX_t$, whose moments are uniformly bounded with respect to time $T,$ that is, for any $p\geq 1,\ T>0,$
\begin{equation}
\label{uniform moments}
\sup_{0\leq t\leq T}\EE\left[ \|\hX_t\|^p \right] \leq C_p,
\end{equation}
and the strong error is also uniformly bounded
\begin{equation}
\label{uniform error}
\sup_{0\leq t\leq T}\EE\left[ \|\hX_t-X_t\|^p \right] \leq c_p\, h^p,
\end{equation}
where $C_p$ and $c_p$ are constants independent of $T,$ and $h$ is the average timestep size.
\end{assumption}
For simplicity, for rest of the paper, we denote the discrete points of a numerical solution by $\hX_{t_n}$ for $n=0,1,...,N,$ and $t_{n+1}=t_{n}+h_n$ which is suitable for both uniform timestep and adaptive schemes. For the numerical experiments on the stochastic Lorenz equation, we use the adaptive Euler-Maruyama method proposed in \cite{Part1,Part2}, in which the uniform bounds for moments and the strong error are achieved for a class of ergodic SDEs. 

\section{Pathwise Sensitivity}
Pathwise sensitivity is a standard approach also known as Infinitesimal Perturbation Analysis (IPA) \cite{LP90}.
Assuming the exchangeability of the integral and the differentiation (Theorem 1 in \cite{LP90} provides a practical and general sufficient validity condition for the exchange of derivative and integral. Fortunately, these conditions are satisfied in most cases.), by Leibniz integral rule and chain rule, we obtain
\[
\frac{\partial F(\theta)}{\partial \theta} \ = \ \EE\left[\frac{\partial\varphi(X_T^{(\theta)})}{\partial \theta}\right]
= \ \EE\left[\left\langle \nabla \varphi(X_T^{(\theta)}),\ \frac{\partial X_T^{(\theta)}}{\partial \theta} \right\rangle\right].
\] 
For $x_T^{(\theta)}\triangleq\frac{\partial X_T^{(\theta)}}{\partial \theta},$ again assuming the exchangeability of the integral and the differentiation (guaranteed by Theorem 39 in chapter 7 in \cite{PP05}), by Leibniz integral rule, we have the variation process
\begin{equation}
\label{Standard Adjoint}
\D\,x_t^{(\theta)} = \left(\frac{\partial f(\theta;X_t^{(\theta)}) }{\partial \theta} + \frac{\partial f(\theta;X_t^{(\theta)})}{\partial X_t^{(\theta)}} x_t^{(\theta)}\right)\, \D t.
\end{equation}
Numerically, we can simulate $X_T^{(\theta)}$ and $x_T^{(\theta)}$ simultaneously:
\begin{eqnarray*}
\hX_{t_{n+1}}^{(\theta)}  &=& \hX_{t_n}^{(\theta)} +\ f(\theta;\hX_{t_n}^{(\theta)})\, h_n\ + \sigma\,\Delta W_n, \\
\hx_{t_{n+1}}^{(\theta)} &=& \hx_{t_n}^{(\theta)} + \left( \frac{\partial f(\theta;\hX_{t_n}^{(\theta)}) }{\partial \theta} + \frac{\partial f(\theta;\hX_{t_n}^{(\theta)})}{\partial \hX_{t_n}^{(\theta)}}\, \hx_{t_n}^{(\theta)}\right)\, h_n.
\end{eqnarray*}
with initial condition $\hX_{0}^{(\theta)}=\xi_0$ and $\hx_0 = 0.$ We then have the standard PS estimator:
\begin{equation}
 \frac{\partial \widehat{F}}{\partial \theta}\ =\ \EE\left[ \left\langle \nabla\varphi(\hX_T^{(\theta)}),\ \hx_T^{(\theta)} \right\rangle\right],
\end{equation}
which is the unbiased estimator of the sensitivity of $\EE\left[\varphi(\hX_T^{(\theta)})\right]$ with respect to $\theta.$
%The key issue of the pathwise sensitivity is the exchangeability of the integral and the differentiation. Here is the Leibniz integral rule in measure theory. 
%\begin{theorem}[Leibniz integral rule]
%Let X be an open subset of  $\mathbf{R}$, and $\Omega$  be a measure space. Suppose $f\colon X\times \Omega \rightarrow \mathbf{R} $ satisfies the following conditions:
%\begin{enumerate}
%\item $f(x,\omega )$ is a Lebesgue-integrable function of $\omega$  for each $x\in X$.
%\item For almost all $\omega \in \Omega$ , the derivative $ f_{x}$ exists for all $x\in X$.
%\item There is an integrable function $\theta \colon \Omega \rightarrow \mathbf {R}$ such that $ |f_{x}(x,\omega )|\leq \theta (\omega )$ for all $x\in X$ and almost every $\omega \in \Omega$.
%\end{enumerate}
%Then, for all $x\in X$,

%\[\frac {\D}{\D x}\int _{\Omega }f(x,\omega )\,\D\omega =\int _{\Omega }f_{x}(x,\omega )\,\D\omega. 
%\]
%\end{theorem} 
The original process $X_t^{(\theta)}$ is ergodic and under suitable conditions, the moments of the numerical solution $X_{t_n}^{(\theta)}$ are uniformly bounded. However, for the variation process $x_t^{\theta}$ of the chaotic system, it may be no longer ergodic and the moments may increase exponentially in $T$. As a result, we have the following theorem.
\vspace{1em}
\begin{theorem}[Standard Pathwise]
\label{Theorem: standard pathwise}
If the moments of the original variation process $x_t$ increases exponentially with respect to time $t,$ then for any numerical solution $\hx_t$ with first order strong convergence, the variance of the standard PS estimator $ \frac{\partial \widehat{F}}{\partial \theta}$ also increases exponentially, that is
\begin{equation}
\label{standard PS variance}
\mathbb{V} \left[ \frac{\partial
 \widehat{F}}{\partial \theta}\right] \leq   \eta\ \e^{\kappa'T},
\end{equation}
for some constant $\eta,\kappa'>0.$ The constant $\kappa'$ is the Lyapunov exponent of the variation process. 
\end{theorem}
\begin{proof}
H{\"o}lder's  inequality gives
\[
\mathbb{V} \left[ \frac{\partial
 \widehat{F}}{\partial \theta}\right] \leq \EE\left[ \left|\frac{\partial
 \widehat{F}}{\partial \theta} \right|^2\right] \leq
 \EE\left[ \left\|\nabla\varphi(\hX_T^{(\theta)}) \right\|^4\right]^{1/2} \EE\left[ \left\| \hx_T^{(\theta)} \right\|^4\right]^{1/2}
\]
The first expectation on the right hand side is bounded uniformly in $T$ due to the Assumption \ref{assp:Numerical solution} and polynomial growth of $\varphi.$ For the second one, the strong convergence of $\hx_t$ implies
\[
\EE\left[ \left\| \hx_T^{(\theta)} \right\|^4\right]
\leq 8\,\EE\left[ \left\| x_T^{(\theta)} \right\|^4 \right] + \lambda,
\]
for some constant $\lambda>0.$ Then the exponential increase of the $\EE\left[ \left\| x_T^{(\theta)} \right\|^4 \right]$ gives the final result.
\end{proof}

If we consider the sensitivity of the outputs related to the invariant measure, we need to choose $T \sim \frac{1}{\lambda^*}\left|\log\varepsilon\right|$ due to \eqref{ConvergenceRateInvariant} to bound the truncation error, which together with Theorem \ref{Theorem: standard pathwise} gives
\[
\mathbb{V} \left[ \frac{\partial
 \widehat{F}}{\partial \theta}\right] \sim   \eta\ \varepsilon^{-\kappa'/\lambda^*}.\]  
Therefore, in order to achieve the $\varepsilon^2$ MSE, we need to simulate $O(\varepsilon^{-2-\kappa'/\lambda^*})$ paths and for each path, we need $h_n=O(\varepsilon)$ to bound the weak error and $T=O(\left| \log \varepsilon\right|).$ The total computational cost of the standard Monte Carlo method becomes $O(\varepsilon^{-3-\kappa'/\lambda^*}|\log\varepsilon|).$
 
\section{Malliavin estimator}
Following the idea in \cite{FLLLT99}, we derive the sensitivity estimator using Malliavin calculus. (The derivation is similar to the proof in Proposition 3.1 in \cite{FLLLT99}, we include it here for clarity and simplicity of this case and it is also similar to our approach to derive the new PS estimator.) Without loss of generality, we assume we are calculating the sensitivity at $\theta=0.$ Under measure $\mathbb{P},$ the original SDE with $\theta=0:$
\begin{equation}
\label{theta 0}
\D X^{(0)}_t = f(X_t^{(0)}) \,\D t + \sigma \,\D W_t^{\mathbb{P}}
\end{equation}
and the perturbed SDE :
\[
\D \Xte_t = \left[ f(\Xte_t) + \theta\, \gamma(\Xte_t)\right] \,\D t + \sigma \,\D W_t^{\mathbb{P}},
\]
where $\gamma:\ \mathbb{R}^m\rightarrow \mathbb{R}^m$ is the perturbation term.
We then consider the expectations
\[F(0) = \EE^{\mathbb{P}}\left[\varphi(X_T^{(0)})\right],\ \ F(\theta)=\EE^{\mathbb{P}}\left[\varphi(\Xte_T)\right],\]
and the derivative by taking the limit
\[
\left.\frac{\partial F(\theta)}{\partial \theta}\right|_{\theta=0}=\lim_{\theta\rightarrow 0}\frac{F(\theta)-F(0)}{\theta}.
\]
Considering a new measure $\mathbb{Q}$ with 
\[
\D W_t^{\mathbb{Q}} = \theta\, \gamma(\Xte_t)/\sigma\,\D t + \D W_t^{\mathbb{P}},
\]
being the standard Brownian motion, we obtain
\[
\D \Xte_t = f(\Xte_t)  \,\D t + \sigma \,\D W_t^{\mathbb{Q}},
\]
under measure $\mathbb{Q}.$ Therefore, by the Girsanov theorem, we have
\[
F(\theta)=\EE^{\mathbb{Q}}\left[\varphi(\Xte_T) R_T(\theta)\right]
,\]
where $R_T(\theta)$ is the Radon-Nikodym derivative 
\[
R_t(\theta) = \frac{\D \mathbb{P}}{\D \mathbb{Q}} = \exp\left( \int_0^t \frac{\theta}{\sigma} \left\langle \gamma(\Xte_u),\,\D W_u^{\mathbb{Q}} \right\rangle -\frac{1}{2} \int_0^t \frac{\theta^2}{\sigma^2} \left\| \gamma(\Xte_u) \right\|^2\, \D u \right).
\]
Since $(\Xte_t,\,W_t^{\mathbb{Q}})$ under measure $\mathbb{Q}$ has the same joint distribution as $(X_t^{(0)},\,W_t^{\mathbb{P}})$ under measure $\mathbb{P},$ we obtain
\[
F(\theta)=\EE^{\mathbb{P}}\left[ \varphi(X_T^{(0)}) R_T(\theta)\right]
\]
with
\[
R_t(\theta) =\exp\left(\int_0^t \frac{\theta}{\sigma} \left\langle \gamma(X_u^{(0)}),\,\D W_u^{\mathbb{P}}\right\rangle -\frac{1}{2} \int_0^t \frac{\theta^2}{\sigma^2} \left\| \gamma(X_u^{(0)}) \right\|^2\, \D u \right).
\]
By the definition of the exponential martingale and It\^{o} formula, we have
\[
\frac{R_T(\theta)-1}{\theta} = \int_0^T \frac{1}{\sigma}\,R_t(\theta)  \left\langle \gamma(X_t^{(0)}),\,\D W_t^{\mathbb{P}} \right\rangle,
\]
and as $\theta \rightarrow 0,$
\[
\frac{R_T(\theta)-1}{\theta} \rightarrow \int_0^T \frac{1}{\sigma} \left\langle \gamma(X_t^{(0)}),\,\D W_t^{\mathbb{P}} \right\rangle\ \ \textbf{in}\ L^2.
\]
Finally, we take the limit $\theta\rightarrow 0$ to get
\[
\frac{ F(\theta) - F(0)}{\theta} = \EE\left[\varphi(X_T^{(0)})\frac{R_T(\theta)-1}{\theta}\right] \rightarrow \EE\left[\varphi(X_T^{(0)}) \int_0^T \left\langle \frac{\gamma(X_t^{(0)})}{\sigma},\  \D W_t^{\mathbb{P}}\right\rangle\right] ,
\]
and obtain the Malliavin estimator:
\begin{equation}
\label{Malliavin estimation}
\left.\frac{\partial F(\theta)}{\partial \theta} \right|_{\theta=0} = \EE\left[\varphi(X_T^{(0)}) \int_0^T \left\langle \frac{\gamma(X_t^{(0)})}{\sigma},\  \D W_t^{\mathbb{P}}\right\rangle\right],
\end{equation}
and numerically
\begin{equation}
\label{Malliavin estimator}
\frac{\partial \widetilde{F}}{\partial \theta} = \varphi(\hX_T) \sum_{n=0}^{N-1} \left\langle \frac{\gamma(\hX_{t_n})}{\sigma},\  \Delta W_n^{\mathbb{P}}\right\rangle.
\end{equation}
The benefit of this approach is the avoidance of the variation process and then the variance of the Malliavin estimator increases only linearly in $T,$ as shown in the following theorem.
\vspace{1em}
\begin{theorem}[Malliavin Estimator]
\label{Theorem: Malliavin Estimator}
If $\hX_t$ satisfies Assumption \ref{assp:Numerical solution} and $\|\gamma(x)\|$ increases at most polynomially when $\|x\|$ is large, the variance of the Malliavin estimator increases linearly with respect to $T,$ that is
\[
\mathbb{V} \left[ \frac{\partial
 \widetilde{F}}{\partial \theta}\right]\ \leq\ \frac{\eta'}{\sigma^2}\, T
\]
for some constant $\eta'>0.$
\end{theorem}
\begin{proof}
By H{\"o}lder's , Burkholder--Davis--Gundy and Jensen's  inequalities and the uniform moments in Assumption \ref{assp:Numerical solution}, we obtain
\begin{eqnarray*}
\mathbb{V} \left[ \frac{\partial
 \widetilde{F}}{\partial \theta}\right] \leq \EE\left[ \left|\frac{\partial
 \widetilde{F}}{\partial \theta}\right|^2\right] &\leq& \EE\left[\left| \varphi(\hX_T)\right|^4\right]^{1/2} \EE\left[\left|\sum_{n=0}^{N-1} \left\langle \frac{\gamma(\hX_{t_n})}{\sigma},\  \Delta W_n^{\mathbb{P}}\right\rangle\right|^4\right]^{1/2}\\
&\leq& \lambda'\, \EE\left[\left| \varphi(\hX_T)\right|^4\right]^{1/2} \EE\left[\left|\sum_{n=0}^{N-1}  \frac{\|\gamma(\hX_{t_n})\|^2}{\sigma^2}h_n \right|^2\right]^{1/2} \leq   \frac{\eta'}{\sigma^2}\, T,
\end{eqnarray*}
for some constant $\lambda',\,\eta'>0.$
\end{proof}

Choosing $T \sim \frac{1}{\lambda^*}\left| \log\varepsilon \right|$ implies 
\[
\mathbb{V} \left[ \frac{\partial
 \widetilde{F}}{\partial \theta}\right] \sim  \frac{\eta'}{\lambda^*\sigma^2} \,\left|\log\varepsilon\right|.\]  
Therefore, in order to achieve the $\varepsilon^2$ MSE, we need to run $O(\varepsilon^{-2} |\log \varepsilon|)$ paths and the total computational cost of the standard Monte Carlo method is $O(\varepsilon^{-3}|\log\varepsilon|^2),$ which is a great improvement over the standard PS method.
However, this benefit is limited to the sensitivity with respect to the drift parameters. For the sensitivity with respect to the volatility parameters and initial condition, the Malliavin estimators again involve the variation process so that the variance of the estimator increases exponentially. Another consideration is the application of the adjoint technique. Since we have to derive different Malliavin estimators for different parameters, we are unable to apply the adjoint technique and then the computational cost increases linearly with the number of sensitivities.

\section{Importance sampling for sensitivity}
\label{New PS section}
In this section, we consider how to use the importance sampling technique to derive the new PS estimator. The main idea is that by introducing a spring term between the original and perturbed SDEs, we obtain a new variation process which is ergodic again. Without loss of generality, we assume we are calculating the sensitivity at $\theta=0.$

Let $\Xo_t$ be the solution to equation \eqref{theta 0} when $\theta=0,$ and consider the new perturbed SDE under measure $\tilde{\mathbb{Q}},$
\[\D Y_t^{(\theta)} \ = \left( f(\theta;Y_t^{(\theta)}) + S\left(\Yo_t-\Yte_t\right) \right)\, \D t\ + \sigma\,\D W_t^{\tilde{\mathbb{Q}}},\]
with 
\[
\D W_t^{\mathbb{P}} = S \left(\Yo_t-\Yte_t\right)/\sigma \, \D t\ + \,\D W_t^{\tilde{\mathbb{Q}}},
\]
which introduces a spring term with sufficiently large $S>0$ to prevent $\left\|\Yo_t\!-\!\Yte_t\right\|$ from increasing exponentially. Note that $\Yte_t = \Xo_t$ and $ W_t^{\tilde{\mathbb{Q}}} =  W_t^{\mathbb{P}}$ when $\theta=0.$

By Girsanov theorem, we have
\[
F(\theta) \ = \ \EE^{\tilde{\mathbb{Q}}}\left[\varphi(Y_T^{(\theta)})R_T(\theta)\right],
\]
where $R_t(\theta)$ is corresponding Radon-Nikodym derivative
\[
R_t(\theta) =  \frac{\D \mathbb{P}}{\D \tilde{\mathbb{Q}}} =\exp\left(-\int_0^t \frac{S}{\sigma} \left\langle \Yo_u-\Yte_u,\,\D W_u^{\tilde{\mathbb{Q}}} \right\rangle -\frac{1}{2} \int_0^t \frac{S^2}{\sigma^2} \left\|\Yo_u-\Yte_u\right\|^2\, \D u \right).
\]
By Leibniz integral rule, we get
\[
\frac{\partial F(\theta)}{\partial \theta} \ = \ \EE^{\tilde{\mathbb{Q}}} \left[\frac{\partial\varphi(Y_T^{(\theta)})}{\partial \theta}R_T(\theta) +\varphi(Y_T^{(\theta)})\frac{\partial R_T(\theta)}{\partial \theta}\right].
\]
and the variation process $\yte_t=\partial \Yte_t/\partial \theta$ satisfies
\[
\D\,\yte_t = \left(\frac{\partial f(\theta;Y_t^{(\theta)}) }{\partial \theta} + \left(\frac{\partial f(\theta;Y_t^{(\theta)}) }{\partial Y_t^{(\theta)}} -SI \right) \yte_t \right)\, \D t,
\]
where $I$ is a $m$-dimensional identity matrix, and
\[
\frac{\partial R_t(\theta)}{\partial \theta} = R_t(\theta)\left( \int_0^t \frac{S}{\sigma} \left\langle \yte_u,\,\D W_u^{\tilde{\mathbb{Q}}} \right\rangle + \int_0^t \frac{S^2}{\sigma^2} \left\langle \Yo_u-\Yte_u, \yte_u\right\rangle\, \D u \right).
\]

Therefore, taking $\theta=0$ and again under measure $\mathbb{P},$
\begin{equation}
\label{New PS estimator}
\left.\frac{\partial F(\theta)}{\partial \theta}\right|_{\theta=0} \ = \ \EE^{\mathbb{P}}\left[ \langle \nabla \varphi(\Yo_T),\, \yo_T\rangle +\varphi(\Yo_T)\int_0^T \frac{S}{\sigma} \left\langle \yo_t,\,\D W_t^{\mathbb{P}} \right\rangle\right].
\end{equation}

Numerically, we only need to simulate $\Yo_t$ and $\yo_t$
\begin{eqnarray}
\label{Change of Measure Greek formula}
\D\, \Yo_t \ &=&   f(0;\Yo_t)   \  \D t\ + \sigma\ \D W_t^{\mathbb{P}}, \nonumber\\
\D\,\yo_t\ &=& \left(\left.\frac{\partial f(\theta;\Yo_t) }{\partial \theta}\right|_{\theta=0} + \left( \left. \frac{\partial f(\theta;Y_t^{(\theta)}) }{\partial Y_t^{(\theta)}}\right|_{\theta=0} -SI \right) \yo_t \right)\, \D t.
\end{eqnarray}
and the new PS estimator is
\begin{equation}
\frac{\partial\breve{F}}{\partial \theta} =
\left\langle \nabla \varphi(\hY^{(0)}_T),\, \hy^{(0)}_T\right\rangle +\varphi(\hY^{(0)}_T)\sum_{n=0}^{N-1} \frac{S}{\sigma} \left\langle \hy_{t_n}^{(0)},\,\Delta W_{n}^{\mathbb{P}} \right\rangle.
\end{equation}
Note that we still simulate the same original process under measure $\mathbb{P}$ with $\Yo_t=X_t$ but a new variation process $\yo_t \neq x_t$. Comparing the two variation processes \eqref{Standard Adjoint} and \eqref{Change of Measure Greek formula}, the new one has an additional linear term which makes the process become ergodic again for a  sufficiently large $S>0.$ By Assumption \ref{assp:Numerical solution}, it is possible to find a suitable numerical scheme such that the moments and strong error of $\hy_t$ are uniformly bounded.
Then, similar to the Malliavin estimator, the variance of the new estimator $\frac{\partial
 \breve{F}}{\partial \theta}$ increases linearly in $T,$ shown in the following theorem.
\vspace{1em}
\begin{theorem}[Importance Sampling]
\label{Theorem: Importance Sampling}
Suppose $\hY_t$ and $\hy_t$ satisfy the Assumption \ref{assp:Numerical solution}, the variance of the new Pathwise estimator with importance sampling increases linearly with respect to $T,$ that is
\[
\mathbb{V} \left[ \frac{\partial
 \breve{F}}{\partial \theta}\right]\ \leq \ \frac{ \eta''}{\sigma^2}\, T,\]   
for some constant $\eta''>0.$ 
\end{theorem}
\begin{proof}
Using the same approach in the proof of Theorem \ref{Theorem: Malliavin Estimator} gives the final result.
\end{proof}

Choosing $T \sim \frac{1}{\lambda^*}|\log\varepsilon|$ implies 
\[
\mathbb{V} \left[ \frac{\partial
 \widetilde{F}}{\partial \theta}\right] \sim  \frac{\eta''}{\lambda^*\sigma} \,|\log\varepsilon|.\]  
Therefore, in order to achieve the $\varepsilon^2$ MSE, we need to run $O(\varepsilon^{-2} |\log \varepsilon|)$ paths and the total computational cost of the standard Monte Carlo method is $O(\varepsilon^{-3}|\log\varepsilon|^2),$ which is the same order as the Malliavin estimators. 

However, one of the additional benefits of this approach is that fundamentally it changes the property of the variation process and we can easily extend to the computation of other sensitivities with respect to volatility parameters and initial condition. 

The other benefit is that the pathwise sensitivities fit into the structure of the adjoint technique naturally which calculates all the sensitivities with a fixed upper bound for the computational cost (up to a factor 4 of the cost for the original SDE) \cite{GG06}.

\section{Numerical results}
\label{Numerical results}
In this section, we present some numerical results for the stochastic Lorenz equation:
\begin{equation}
\label{Lorenz example}
\begin{aligned}
\D\, X_t \ = 
\begin{pmatrix}
10(X_t^2-X_t^1)\\
X_t^1(\theta-X_t^3) -X_t^2\\
X_t^1X_t^2-\frac{8}{3}X_t^3
\end{pmatrix}
\D t + \sigma\ \D W_t^{\mathbb{P}},
\end{aligned}
\end{equation}
with initial condition $x_0 = (-2.4, -3.7, 14.98)^T$ and $\sigma=6$. We estimate the sensitivity of $\EE\left[X_T^3\right]$ with respect to $\theta$ at $\theta=28.$ 

For the standard pathwise sensitivity method, we directly derive the variation process following \eqref{Standard Adjoint}:
\begin{equation*}
\begin{aligned}
\D\, x_t =\ \left(
\begin{pmatrix}
0\\
X_t^1\\
0
\end{pmatrix}
+
\begin{pmatrix}
-10 & 10 & \\
28-X_t^3 & -1 & -X_t^1\\
X_t^2 & X_t^1 & -\frac{8}{3}
\end{pmatrix}
x_t
\right) \D t 
\end{aligned} 
\end{equation*}
and the standard PS estimator:
\begin{equation}
 \frac{\partial \widehat{F}}{\partial \theta}\ =\ \hx^3_T.
\end{equation}
Numerically, we set $T=20$ and $\sigma =6$ using the same adaptive timestep function in \cite{Part3} with $\delta=2^{-9}$ and simulate $N=4\times10^7$ paths. Then we plot the log variance of the estimator $\frac{\partial
 \widehat{F}}{\partial \theta}$ with respect to $T$ in Figure \ref{ExponStdGreek}. 
%\begin{figure}[H]
%\includegraphics[width=0.9\textwidth]{LorenzProblem/Exp_std_T}
%\end{figure}
%\begin{figure}[H]
%\includegraphics[width=0.9\textwidth]{LorenzProblem/Exp_std_T_log}
%\end{figure}
\begin{figure}[h]
\center
\includegraphics[width=0.7\textwidth]{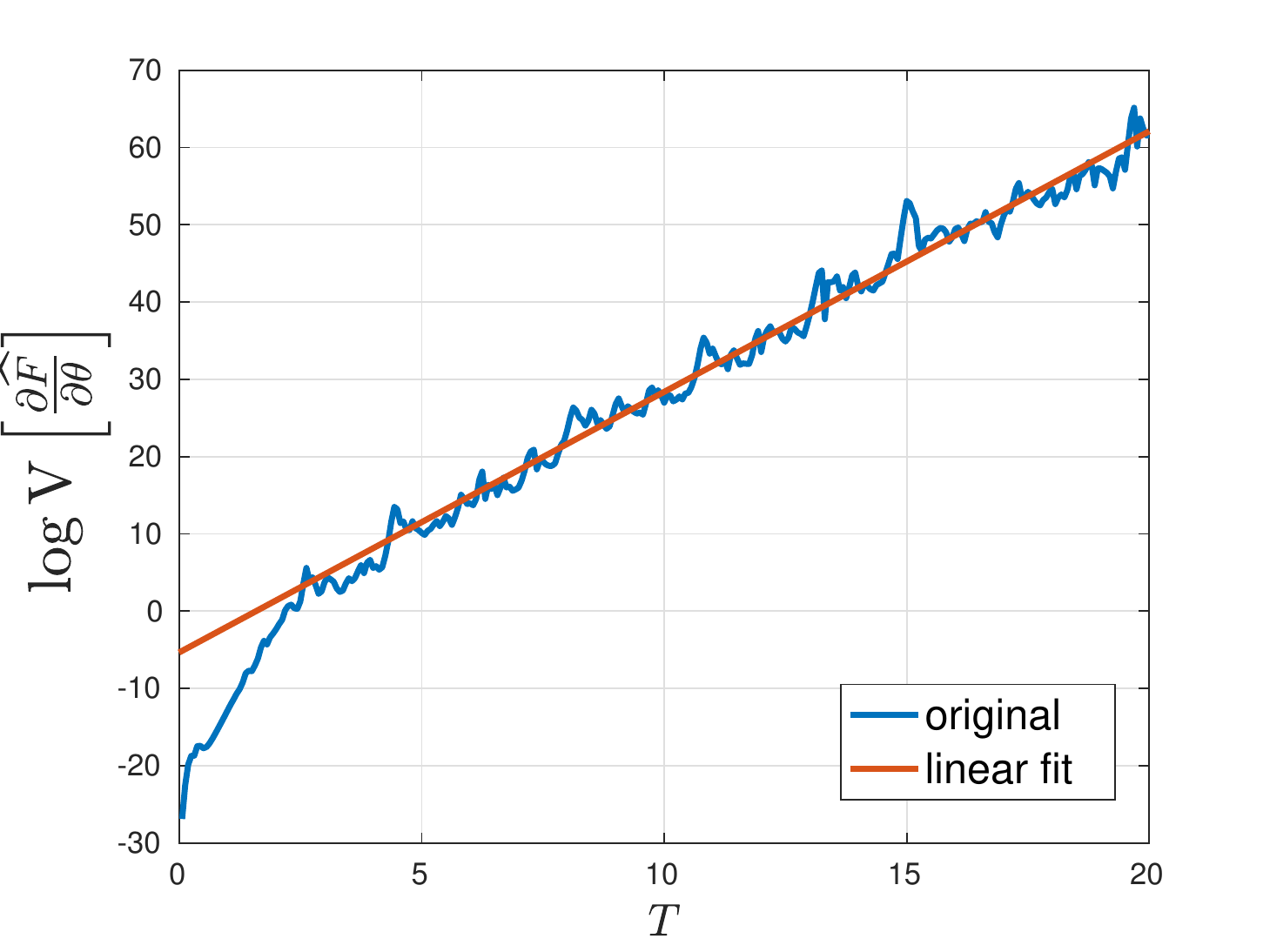}
\caption{Exponential increase of the variance of standard PS estimator}
\label{ExponStdGreek}
\end{figure}
It shows clearly that the log variance increase linearly in $T$ which implies the variance itself increases exponentially. The slope we fitted is $3.37$ which is an approximation of $\kappa'/\lambda^*$ and indicates that the computational cost the standard Monte Carlo method using standard PS estimator to achieve $O(\varepsilon^2)$ MSE becomes $O(\varepsilon^{-5.37}|\log \varepsilon|).$

Next, following \eqref{Malliavin estimation} and \eqref{Malliavin estimator}, we have $\gamma(X_t) = ( 0, X_t^1, 0)^T$ and get the Malliavin estimation
\[
\frac{\partial F(\theta)}{\partial \theta} = \EE^{\mathbb{P}}\left[ X_T^3 \int_0^T \frac{1}{\sigma} X_t^1\, \D W^{\mathbb{P},2}_t\right],
\]
and the Malliavin estimator
\[
\frac{\partial \widetilde{F}}{\partial \theta} = \hX^3_T \sum_{n=0}^{N-1} \hX^1_{t_n} \Delta W_n^{\mathbb{P},2}.
\]
Finally, following \eqref{Change of Measure Greek formula}, we obtain the new SDEs with importance sampling,
\begin{equation*}
\begin{aligned}
\D\, Y_t \ = 
\begin{pmatrix}
10(Y_t^2-Y_t^1)\\
Y_t^1(28-Y_t^3) -Y_t^2\\
Y_t^1Y_t^2-\frac{8}{3}Y_t^3
\end{pmatrix}
\D t + \sigma\ \D W_t^{\mathbb{P}},
\end{aligned} 
\end{equation*}
and
\begin{equation*}
\begin{aligned}
\D\, y_t \ = \left(
\begin{pmatrix}
0\\
Y_t^1\\
0
\end{pmatrix}
+
\begin{pmatrix}
-10-S & 10 & \\
28-Y_t^3 & -1-S & -Y_t^1\\
Y_t^2 & Y_t^1 & -\frac{8}{3}-S
\end{pmatrix}
y_t
\right) \D t, 
\end{aligned} 
\end{equation*}
for some $S>0.$ We get the new PS expression
\begin{equation}
\frac{\partial F(\theta)}{\partial \theta} =  \ \EE^{\mathbb{P}}\left[ y_T^3 +Y_T^3\int_0^T \frac{S}{\sigma} \left\langle y_t,\,\D W_t^{\mathbb{P}} \right\rangle\right],
\end{equation}
and the new PS estimator
\begin{equation}
\label{COM Lorenz estimator}
\frac{\partial\breve{F}}{\partial \theta} =
\hy^3_T+\hY^3_T \sum_{n=0}^{N-1} \frac{S}{\sigma} \left\langle \hy_{t_n},\,\Delta W_{n}^{\mathbb{P}} \right\rangle.
\end{equation}
Numerically, we again set same $T,\ h,\ \sigma$ and $N$ as above and choose $S=10.$ Then we plot the variances of the Mallianvin and new PS estimators with respect to $T$ in Figure \ref{LinearCOMGreek}. 
%\begin{figure}[H]
%\includegraphics[width=0.9\textwidth]{LorenzProblem/Exp_com_T}
%\end{figure}
\begin{figure}[h]
\center
\subfigure[Comparison of Variances]{
\label{LinearCOMGreek}
\includegraphics[width=0.47\textwidth]{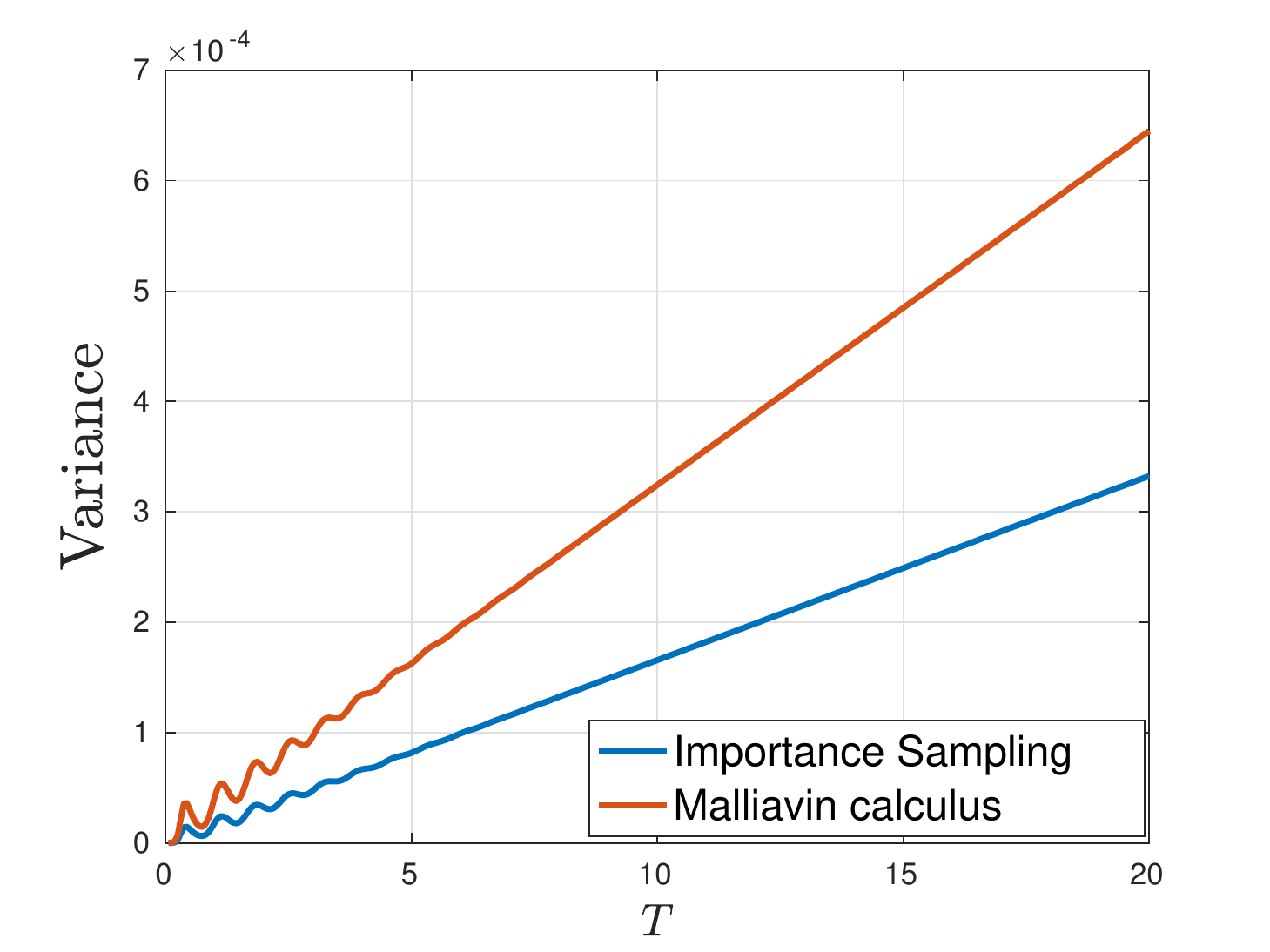}
}
\subfigure[Comparison of Expectations]{
\label{COM vs Malli}
\includegraphics[width=0.47\textwidth]{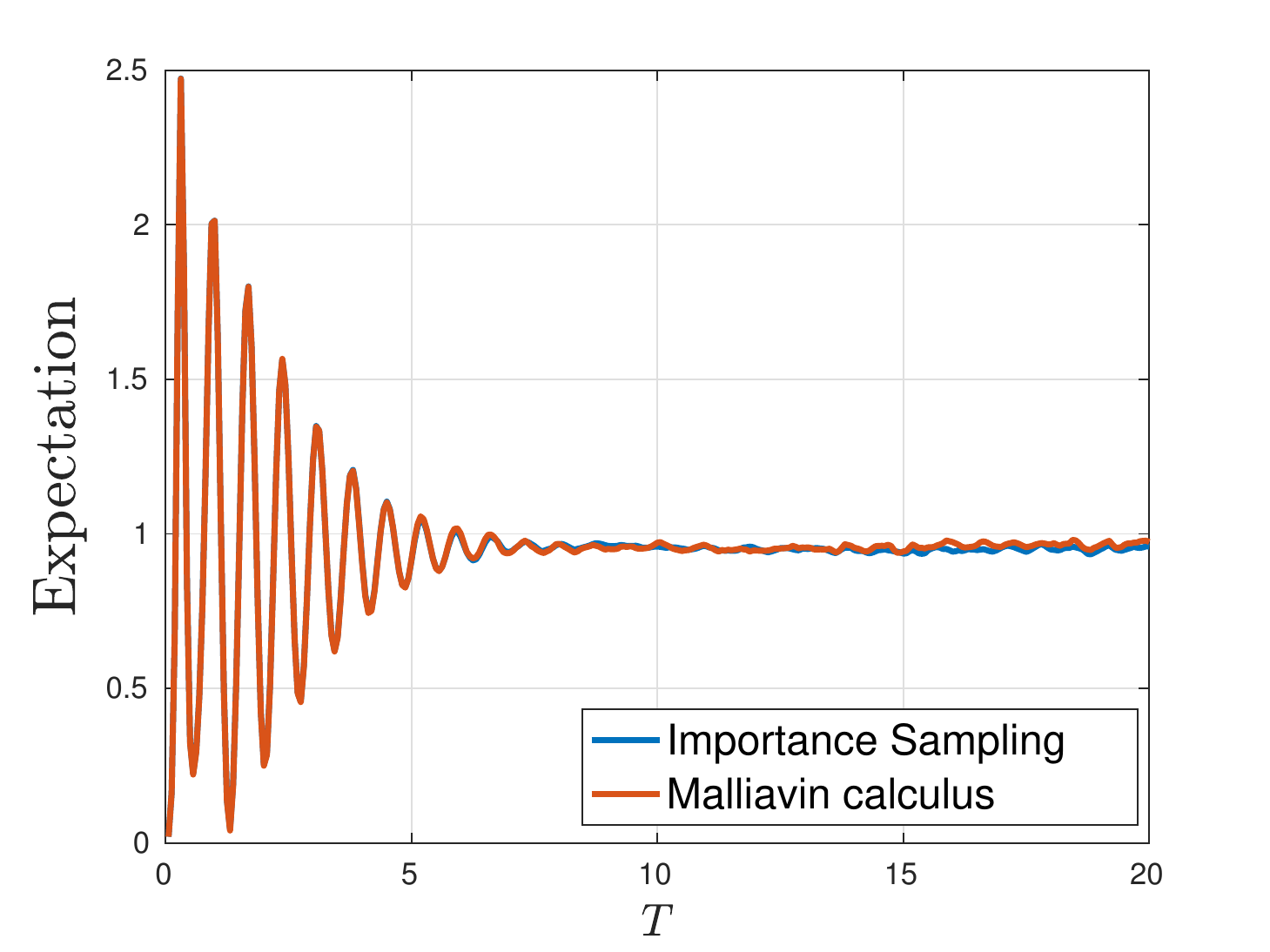}
}
\caption{Comparison of Malliavin and new PS (Importance Sampling) estimators}
\end{figure}

%\begin{figure}[H]
%\center
%\includegraphics[width=0.7\textwidth]{LorenzProblem/Linear_increase_variance}
%\caption{Linear increase of the variance}
%\label{LinearCOMGreek}
%\end{figure}
Figure \ref{LinearCOMGreek} shows clearly the variances of both Malliavin and new PS estimators increase linearly in $T,$ but the variance of the new PS estimator has a smaller constant factor and is slightly more efficient. Figure \ref{COM vs Malli} shows that both methods give similar estimates and confirms that the estimators are ergodic and converge to the sensitivity under the invariant measure.

%\begin{figure}[H]
%\center
%\includegraphics[width=0.7\textwidth]{LorenzProblem/COM_Mali_11}
%\caption{Comparison of Expectations}
%\label{COM vs Malli}
%\end{figure}

%\begin{figure}[H]
%\center
%\includegraphics[width=0.7\textwidth]{LorenzProblem/COM_Mali_12}
%\caption{Comparison of variance}
%\label{COM vs Malli 2}
%\end{figure}

%\begin{figure}[H]
%\center
%\includegraphics[width=0.7\textwidth]{LorenzProblem/COM_Mali}
%\caption{Comparison of value}
%\label{COM vs Malli}
%\end{figure}

%\begin{figure}[H]
%\center
%\includegraphics[width=0.7\textwidth]{LorenzProblem/COM_Mali2}
%\caption{Comparison of variance}
%\label{COM vs Malli 2}
%\end{figure}

\section{MLMC for sensitivity}
In this section, we extend both the Malliavin and the new PS estimators to MLMC \cite{giles08,giles15}. For ergodic SDEs without contractivity, we need to employ the change of measure technique introduced in \cite{Part3} to add a spring term between the fine and coarse paths, since otherwise the fine and coarse paths may diverge exponentially due to the chaotic property. We first quickly review this technique and then present the numerical results.

\subsection{MLMC with change of measure}
Instead of considering the fine and coarse paths of the original SDEs under the same measure $\mathbb{P}:$
\begin{eqnarray*}
\label{Eq: SDE before change}
\D X_t^f &=& f(X_t^f)\,\D t\  +\ \sigma\, \D W_t^{\mathbb{P}},\nonumber\\
\D X_t^c &=& f(X_t^c)\,\D t\ +\ \sigma\, \D W_t^{\mathbb{P}}.
\end{eqnarray*}
We add a spring term with spring coefficient $S>0$ for both fine path and coarse paths and consider both paths under different measures:
\begin{eqnarray*}
\label{Eq: SDE for change}
\mathbb{Q}^f:\ \ \D Y_t^f &=&  f(Y_t^f)\,\D t + \sigma\,\D W_t^{\mathbb{Q}^f},\nonumber\\
\mathbb{Q}^c:\ \ \D Y_t^c &=&  f(Y_t^c)\,\D t + \sigma\,\D W_t^{\mathbb{Q}^c},
\end{eqnarray*}
with
\begin{eqnarray}
\label{COM Fine Coarse W}
\D W_t^{\mathbb{Q}^f} &=& \frac{S}{\sigma}(Y_t^c-Y_t^f)\,\D t+ \D W_t^{\mathbb{P}},\nonumber\\
\D W_t^{\mathbb{Q}^c} &=& \frac{S}{\sigma}(Y_t^f-Y_t^c)\,\D t + \D W_t^{\mathbb{P}}.
\end{eqnarray}
Therefore, under simulation measure $\mathbb{P},$ we obtain
\begin{eqnarray*}
\label{Eq: SDE for change}
\D Y_t^f &=&  S(Y_t^c-Y_t^f)\,\D t +f(Y_t^f)\,\D t + \sigma\,\D W_t^{\mathbb{P}},\nonumber\\
\D Y_t^c &=&   S(Y_t^f-Y_t^c)\,\D t+ f(Y_t^c)\,\D t +\sigma\,\D W_t^{\mathbb{P}}.
\end{eqnarray*}
The Girsanov theorem gives
\begin{equation}
\label{MLMC level}
\EE^{\mathbb{P}}[\varphi(X_T^f)]-\EE^{\mathbb{P}}[\varphi(X_T^c)] =\EE^{\mathbb{Q}^f}[\varphi(Y_T^f)]-\EE^{\mathbb{Q}^c}[\varphi(Y_T^c)]=
\EE^{\mathbb{P}}\left[\varphi(Y_T^f)\frac{\D \mathbb{Q}^f}{\D \mathbb{P}_T}-\varphi(Y_T^c)\frac{\D \mathbb{Q}^c}{\D \mathbb{P}_T}\right],
\end{equation}
where $\frac{\D \mathbb{Q}^f}{\D \mathbb{P}_T}$ is the corresponding Radon-Nikodym derivative with the following form:
\begin{equation}
\label{MLMC RD}
\frac{\D \mathbb{Q}^f}{\D \mathbb{P}_T} = \exp\left(-\int_0^T \left\langle \frac{S}{\sigma}(Y_t^f-Y_t^c),\,\D W_t^\mathbb{P} \right\rangle - \frac{1}{2} \int_0^T \frac{S^2}{\sigma^2}\left\|Y_t^f-Y_t^c\right\|^2\,\D t \right)
\end{equation}
and $\frac{\D \mathbb{Q}^c}{\D \mathbb{P}_T}$ is similar.
 The benefit of this technique is that under measure $\mathbb{P},$ we recover the contractivity between $Y_t^c$ and $Y_t^f$ and the variance of the level estimator increases linearly in $T$ instead of exponentially. Note that for the numerical implementation, we derive the exact Radon-Nikodym derivative for the numerical solution instead of the numerical approximation of \eqref{MLMC RD}. For the detailed numerical scheme, see section 2 in \cite{Part3}.

However, in this paper, we observe that both Malliavin and new PS estimations \eqref{Malliavin estimation} and \eqref{New PS estimator} involve It\^{o} integrals and we need to be careful when performing the change of measure and respect the following identity:
\[
\EE^{\mathbb{Q}^f}\left[\int_0^T \left\langle \psi(Y_t^f),\  \D W_t^{\mathbb{Q}^f}\right\rangle\right]
=
\EE^{\mathbb{P}}\left[\int_0^T \left\langle \psi(Y_t^f),\  \frac{S}{\sigma}(Y_t^c-Y_t^f)\,\D t+ \D W_t^{\mathbb{P}}\right\rangle \frac{\D \mathbb{Q}^f}{\D \mathbb{P}_T}\,\right],
\]
where the It\^{o} integral is viewed as a functional of the whole path. The identity for the coarse path is similar. Therefore, for numerical experiments, we need to use $W_t^{\mathbb{P}}$ to reconstruct $W_t^{\mathbb{Q}^f}$ and $W_t^{\mathbb{Q}^c}$ to calculate the integral following \eqref{COM Fine Coarse W}.
\vspace{1em}
\begin{theorem}[Level Estimators with Change of Measure]
\label{Theorem: MLMC COM}
Suppose $\hY_t$ and $\hy_t$ satisfy the Assumption \ref{assp:Numerical solution} under measure $\mathbb{P}$ and $f$ and $\gamma$ are globally Lipschitz, the variance of the new MLMC level estimator with change of measure \eqref{MLMC level} and timestep $h$ increases linearly with respect to $T^2,$ that is
\[
\mathbb{V}\left[
\tilde{\varphi}(\widehat{Y}^f_T)\, \frac{\D \widehat{\QQ}^f}{\D \PP_T} -
\tilde{\varphi}(\widehat{Y}^c_T)\, \frac{\D \widehat{\QQ}^c}{\D \PP_T} 
\right] 
\leq \frac{\eta'''}{\sigma^4}\, T^2\, h^2.
\]  
for some constant $\eta'''>0,$ where $\tilde{\varphi}$ is the functional of the Malliavin or new PS estimator. 
\end{theorem}
\begin{proof}
Using the same approach in the proof of Theorem 4 in \cite{Part3} and the fact that $\EE\left[|\tilde{\varphi}(\widehat{Y}^f_T)|^p\right] \leq  \lambda''/\sigma^{p}\, T^{p/2}$ gives the final result. Intuitively, the first order strong convergence gives $h^2,$ the It\^{o} integral gives a factor $T/\sigma^2$ by previous Theorems and another factor of $T/\sigma^2$ comes from the Radon-Nikodym derivative \eqref{MLMC RD}.
\end{proof}

The MLMC theorem in \cite{Part3} applies here except that $V_0=O(T/\sigma^2).$ By choosing $T=O(|\log\varepsilon|)$ we have $V_0=O(|\log \varepsilon|),$ and the optimal computational cost of the MLMC with change of measure to achieve $O(\varepsilon^2)$ MSE becomes $O(\varepsilon^{-2}|\log\varepsilon|^3).$ Compared with the results in \cite{Part3}, the additional log term comes from the order of $V_0.$

\subsection{Numerical results}
In this subsection, we discuss the applications of standard MLMC and MLMC with change of measure to both Malliavin estimator \eqref{Malliavin estimation} and new PS estimator \eqref{New PS estimator}. Numerically, we use an adaptive timestep with $h_0=2^{-7}$ and $N=10^5$ to compute the variances $V_0$ on level 0 and $V_1$ level 1 with respect to $T$ for stochastic Lorenz equation with $\sigma=6.$

First, we directly use the standard MLMC scheme for both estimators without change of measure. The variance on level $0$ increases linearly in $T$ which is the same as the standard Monte Carlo method. However, for level $\ell\geq 1,$ the variance of level estimator still increases exponentially due to the chaotic property of the original SDEs. 

\begin{figure}[H]
\center
\subfigure[Malliavin estimator]{
\includegraphics[width=0.45\textwidth]{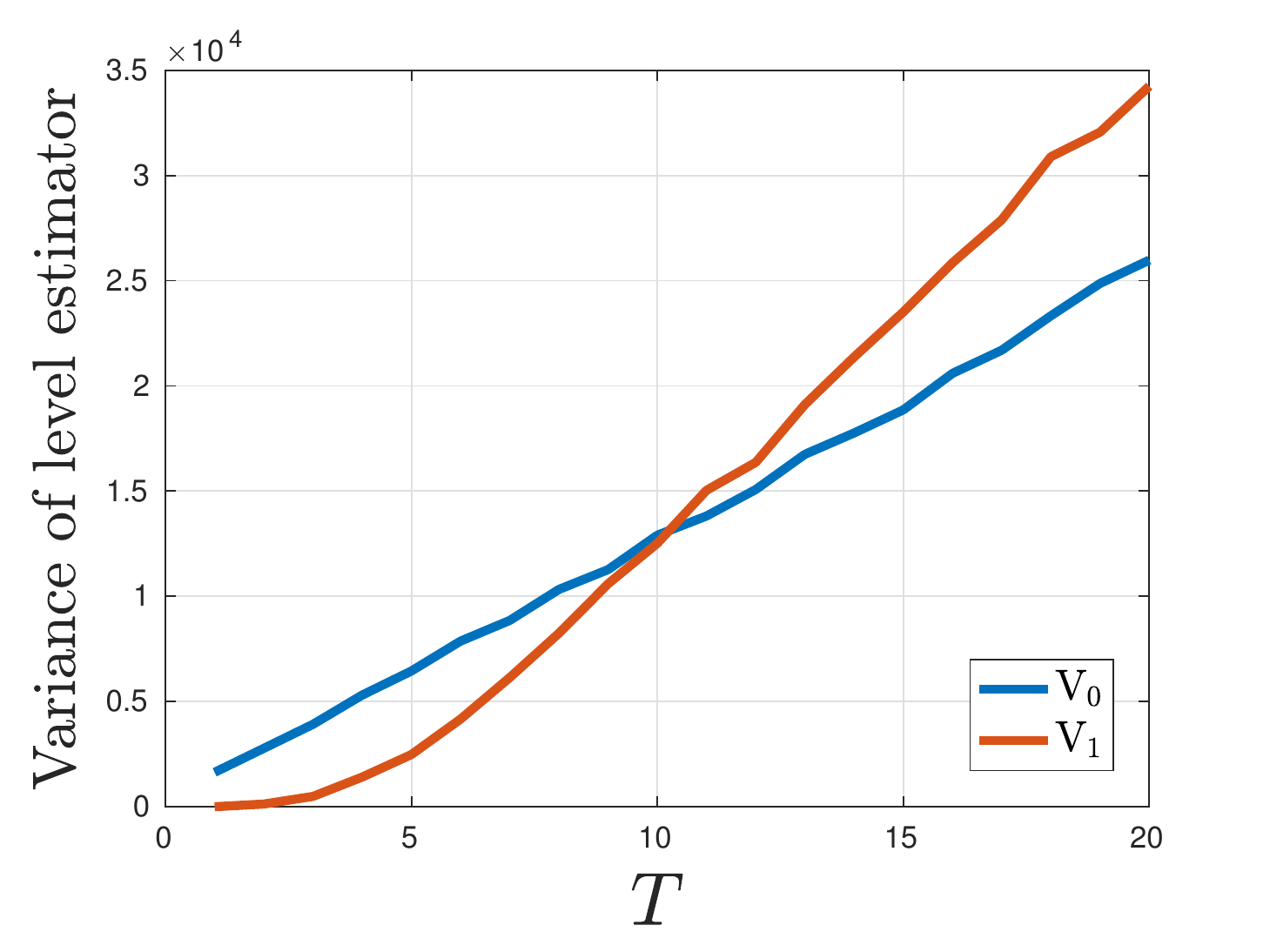}
}
\subfigure[New PS estimator]{
\includegraphics[width=0.45\textwidth]{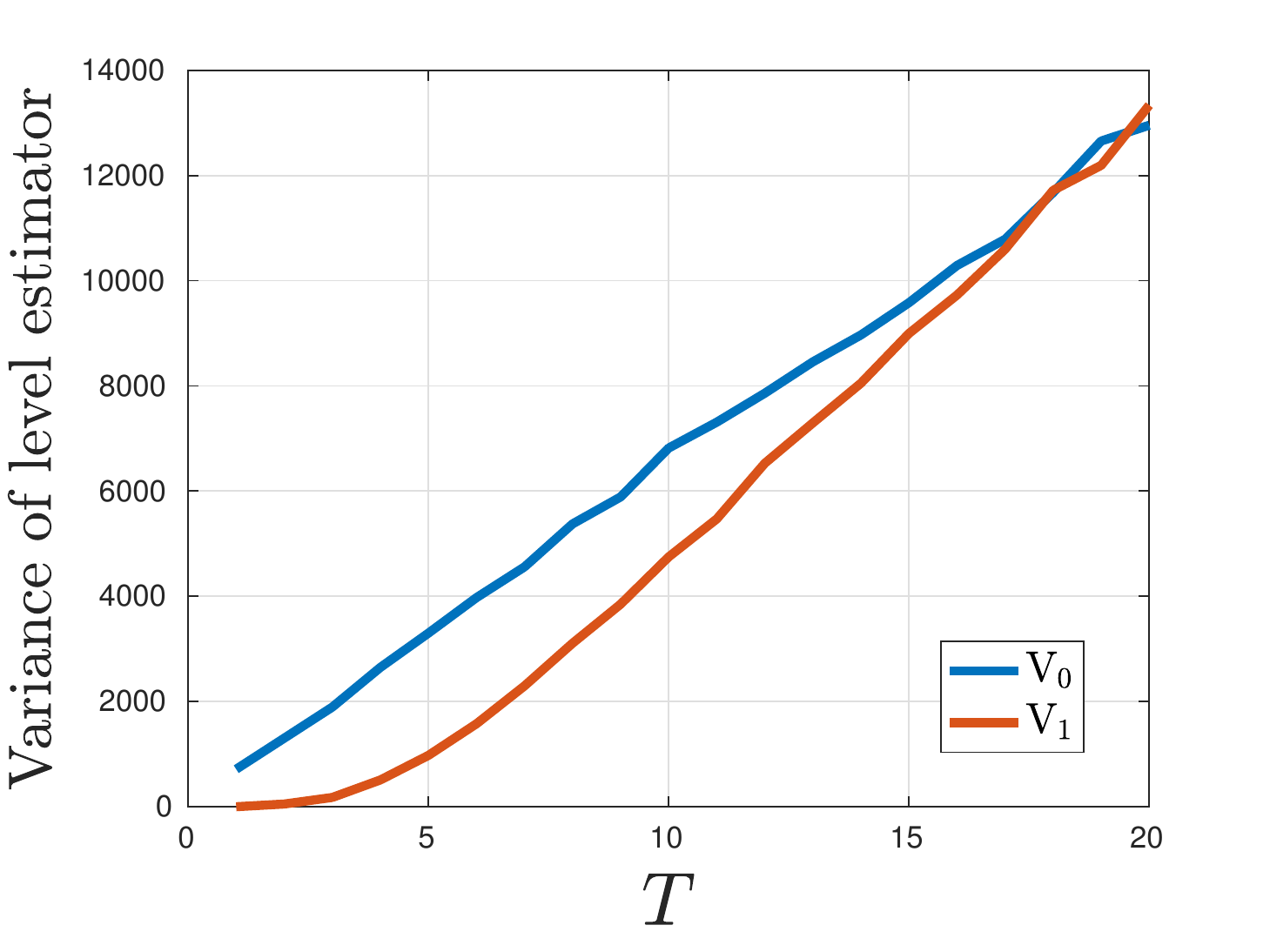}
}
\caption{MLMC variances on level 0 and 1 without change of measure}
\label{MLMC V}
\end{figure}

In Figure \ref{MLMC V}, we plot the $V_0$ and $V_1$ with respect to $T$ and it confirms the linear increase of $V_0$ and the initial exponential increase of $V_1.$ For both estimators, $V_1$ becomes larger than $V_0$ as $T$ increases due to the bad coupling between fine and coarse paths. 

Second, we add the spring term between the fine and coarse paths and perform the change of measure on the level estimator following \eqref{MLMC level} with $S=10$. The variance on level $0$ increases linearly. For level $\ell\geq1,$ the variance of level estimator increases only quadratically as shown in Theorem \ref{Theorem: MLMC COM}. 

\begin{figure}[H]
\center
\subfigure[Malliavin estimator]{
\includegraphics[width=0.45\textwidth]{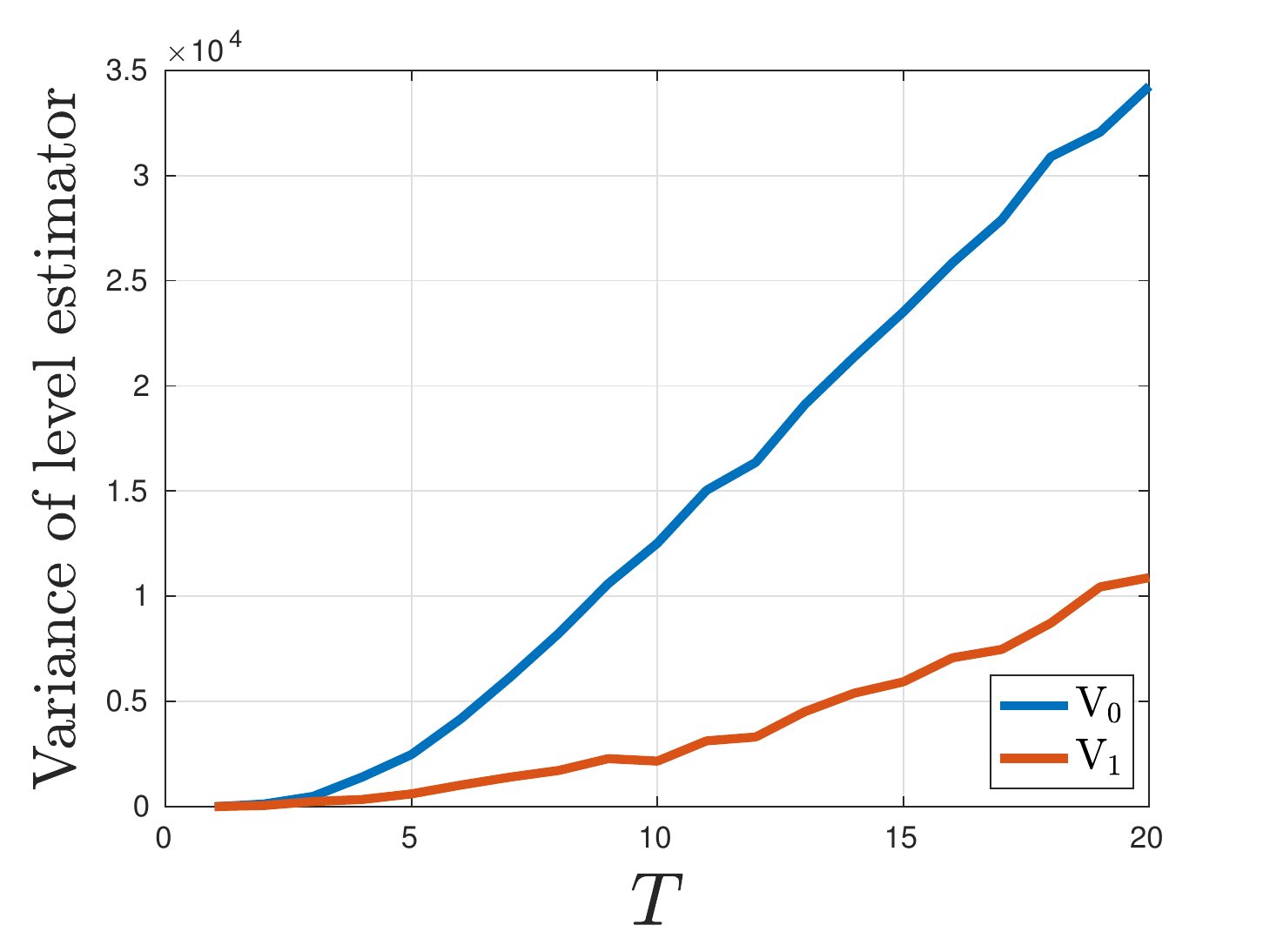}
}
\subfigure[New PS estimator]{
\includegraphics[width=0.45\textwidth]{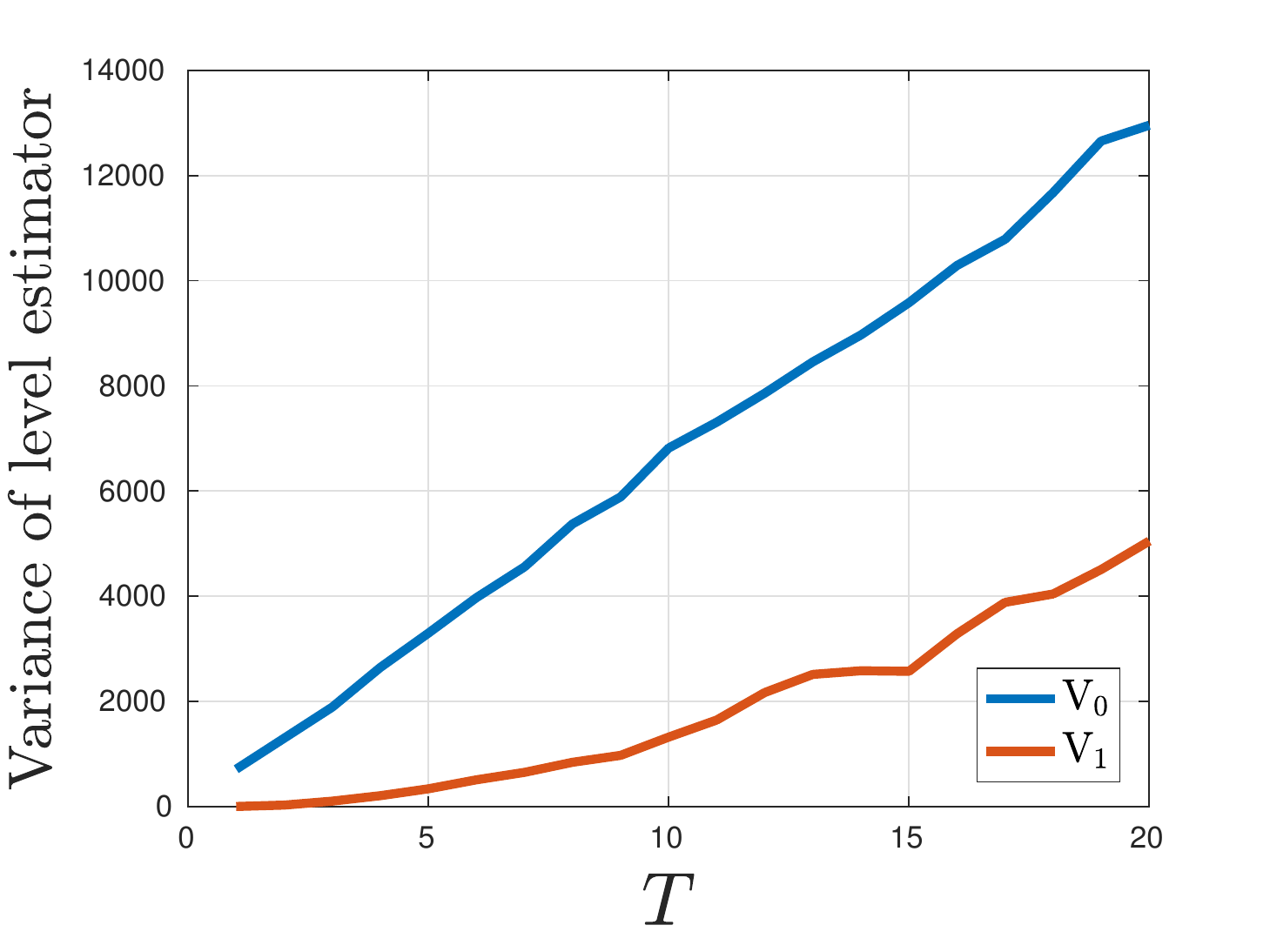}
}
\caption{MLMC variances on level 0 and 1 with change of measure}
\label{MLMC COM V}
\end{figure}

In Figure \ref{MLMC COM V}, we plot the $V_0$ and $V_1$ with respect to $T$ and it confirms the linear increase of $V_0$ and the good coupling on level 1. The quadratic increases with respect to $T$ are shown in Figure \ref{MLMC-V1-T2} by plotting $V_1$ with respect to $T^2.$

\begin{figure}[H]
\center
\subfigure[Malliavin estimator]{
\includegraphics[width=0.45\textwidth]{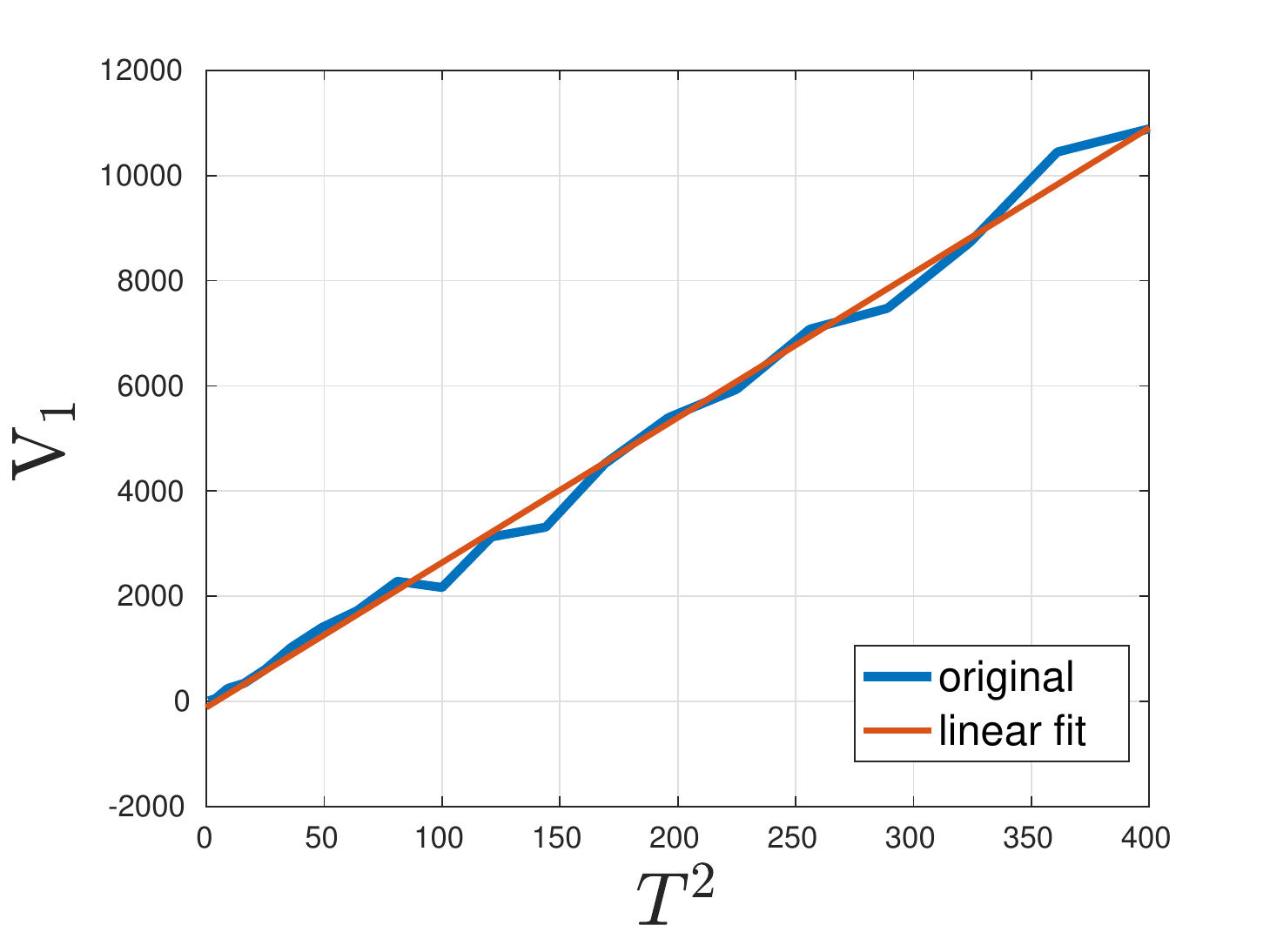}
}
\subfigure[New PS estimator]{
\includegraphics[width=0.45\textwidth]{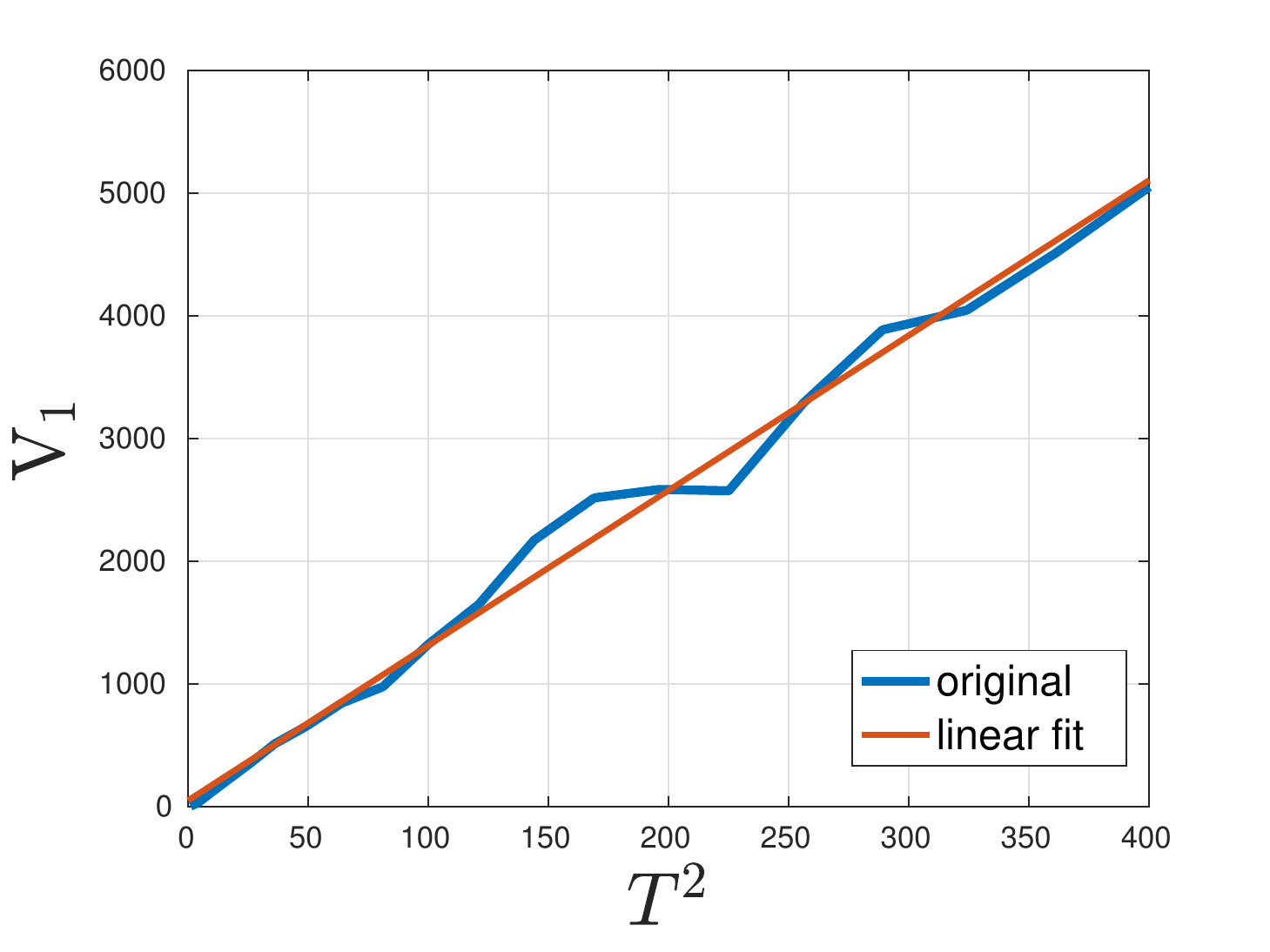}
}
\caption{Quadratic increase with respect to $T$}
\label{MLMC-V1-T2}
\end{figure}

Finally, we run the MLMC scheme for both estimators and choose $T=10$ which is sufficiently large for acceptable weak convergence by Figure \ref{COM vs Malli}.

\begin{figure}[H]
\center
\includegraphics[width=0.77\textwidth]{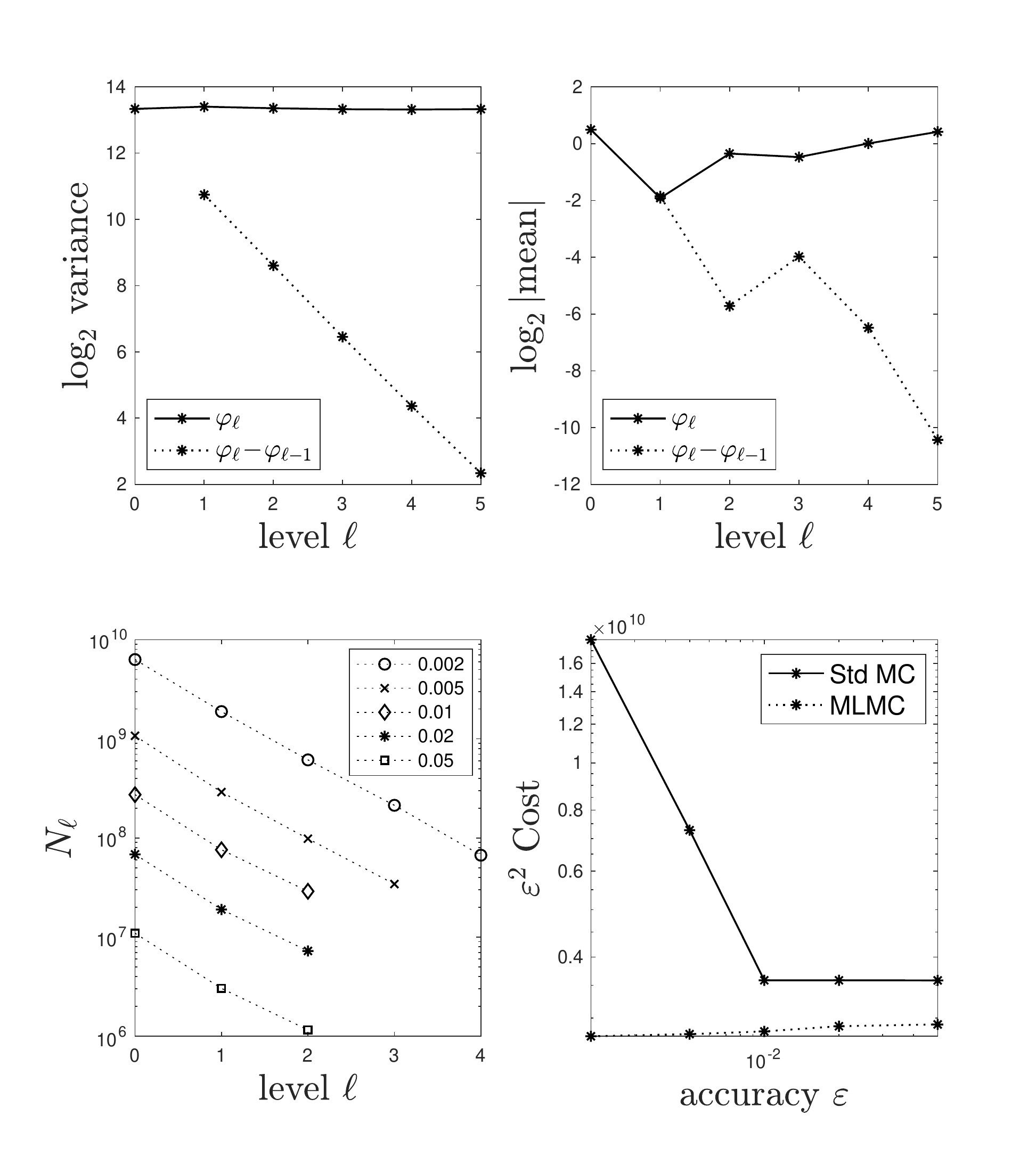}
\caption{MLMC for Malliavin estimator}
\label{MLMC-Malli-sum}
\end{figure}

Figures \ref{MLMC-Malli-sum} and \ref{MLMC-PS-sum} show the MLMC for both Malliavin and new PS estimators works well. The top left plot shows that the variance of level estimator decreases at rate 2, i.e. the variance is proportional to $2^{-2\ell}$ corresponding to the first order strong convergence. The top right plot shows the first order weak convergence. The bottom left plot shows the different number of paths on different level for different accuracy. The bottom right plot shows that the computational cost for MLMC is $O(\varepsilon^{-2})$ and standard Monte Carlo is $O(\varepsilon^{-3}).$

\begin{figure}[H]
\center
\includegraphics[width=0.8\textwidth]{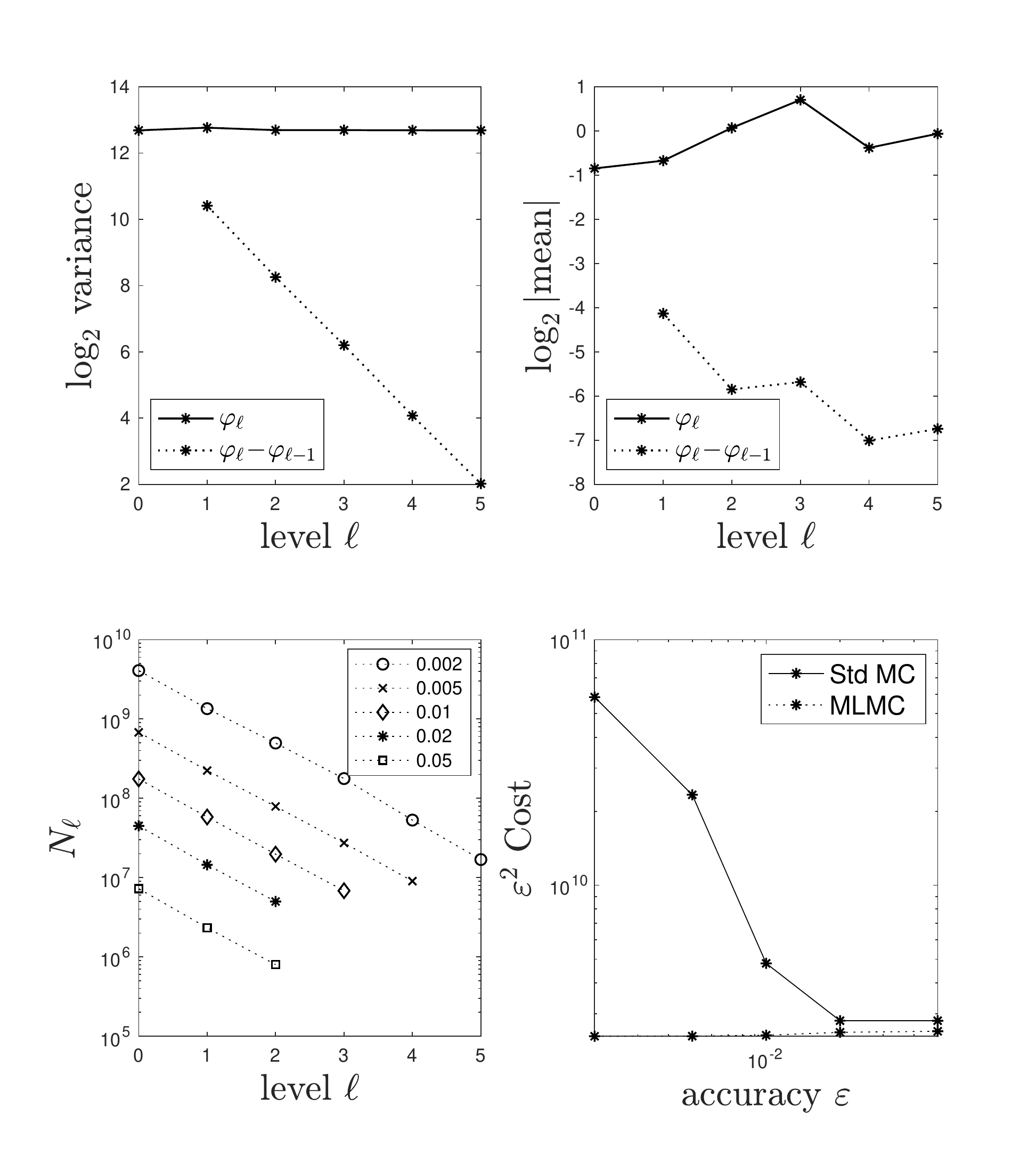}
\caption{MLMC for new PS estimator}
\label{MLMC-PS-sum}
\end{figure}

\section{Other sensitivities}
In this section, we consider the sensitivities of the invariant measure with respect to the volatility parameter $\sigma$ and initial condition $\xi_0;$ the latter should be $0$. Note that the Malliavin estimators fail here since they involve the variation process shown in Proposition 3.3 in \cite{FLLLT99}. 

Assume we are calculating the sensitivity at $\sigma_0$ and consider the perturbed volatility $\sigma.$ The derivation is similar to Section \ref{New PS section} except that we have the new variation process with respect to $\sigma:$
\[
\D\,\zts_t =  \left(\frac{\partial f(Y_t^{(\sigma)}) }{\partial Y_t^{(\sigma)}} -SI \right) \zts_t \, \D t + \D W_t^{\tilde{\mathbb{Q}}},
\]
and the derivative of Radon-Nikodym derivative with respect to $\sigma,$
\begin{eqnarray*}
\frac{\partial \Rts_t}{\partial \sigma} &=& \Rts_t\left( \int_0^t \frac{S}{\sigma} \left\langle \zts_u,\,\D W_u^{\tilde{\mathbb{Q}}} \right\rangle +\int_0^t \frac{S}{\sigma^2} \left\langle Y_u^{(\sigma_0)}-\Yts_u,\,\D W_u^{\tilde{\mathbb{Q}}} \right\rangle \right.\\
&&\left.\ \ \ \ \ \ \ \ \ + \int_0^t \frac{S^2}{\sigma^2} \left\langle Y_u^{(\sigma_0)}-\Yts_u, \zts_u\right\rangle\, \D u + \int_0^t \frac{S^2}{\sigma^3} \left\|Y_u^{(\sigma_0)}-\Yts_u\right\|^2\, \D u \right).
\end{eqnarray*}

Therefore, taking $\sigma=\sigma_0$ and again under measure $\mathbb{P},$
\begin{equation}
\label{New PS vol estimator}
\left.\frac{\partial F(\sigma)}{\partial \sigma}\right|_{\sigma=\sigma_0} \ = \ \EE^{\mathbb{P}}\left[ \langle \nabla \varphi(Y^{(\sigma_0)}_T),\, z_T^{(\sigma_0)}\rangle +\varphi(Y^{(\sigma_0)}_T)\int_0^T \frac{S}{\sigma} \left\langle z_t^{(\sigma_0)},\,\D W_t^{\mathbb{P}} \right\rangle\right].
\end{equation}

Numerically, we only need to simulate $Y^{(\sigma_0)}_t$ and $y^{(\sigma_0)}_t$
\begin{eqnarray}
\label{Volatility Greek formula}
\D\, Y^{(\sigma_0)}_t \ &=&   f(Y^{(\sigma_0)}_t)   \  \D t\ + \sigma_0\ \D W_t^{\mathbb{P}}, \nonumber\\
\D\,z^{(\sigma_0)}_t\ &=&   \left(\frac{\partial f(Y_t^{(\sigma_0)}) }{\partial Y_t^{(\sigma_0)}} -SI \right) z^{(\sigma_0)}_t \, \D t \ + \ \D W_t^{\mathbb{P}},
\end{eqnarray}
and the new PS estimator is
\begin{equation}
\frac{\partial\breve{F}}{\partial \sigma} =
\left\langle \nabla \varphi(\hY^{(\sigma)}_T)\, , \hz^{(\sigma)}_T \right\rangle +\varphi(\hY^{(\sigma)}_T)\sum_{n=0}^{N-1} \frac{S}{\sigma} \left\langle \hz_{t_n}^{(\sigma)},\,\Delta W_{n}^{\mathbb{P}} \right\rangle.
\end{equation}
We perform the numerical experiment on the stochastic Lorenz equation with $\sigma=6.$ Figure \ref{Sigma Sen} shows the estimator is ergodic and the expectation converges to the sensitivity of the invariant measure 0.1274, which is close to the finite difference value 0.1268 using the data in Figure \ref{yitong4}. Figure \ref{Sigma Sen Var} confirms the linear increase of the variance.

\begin{figure}[h]
\center
\subfigure[Expectation with respect to $T$ ]{
\label{Sigma Sen}
\includegraphics[width=0.45\textwidth]{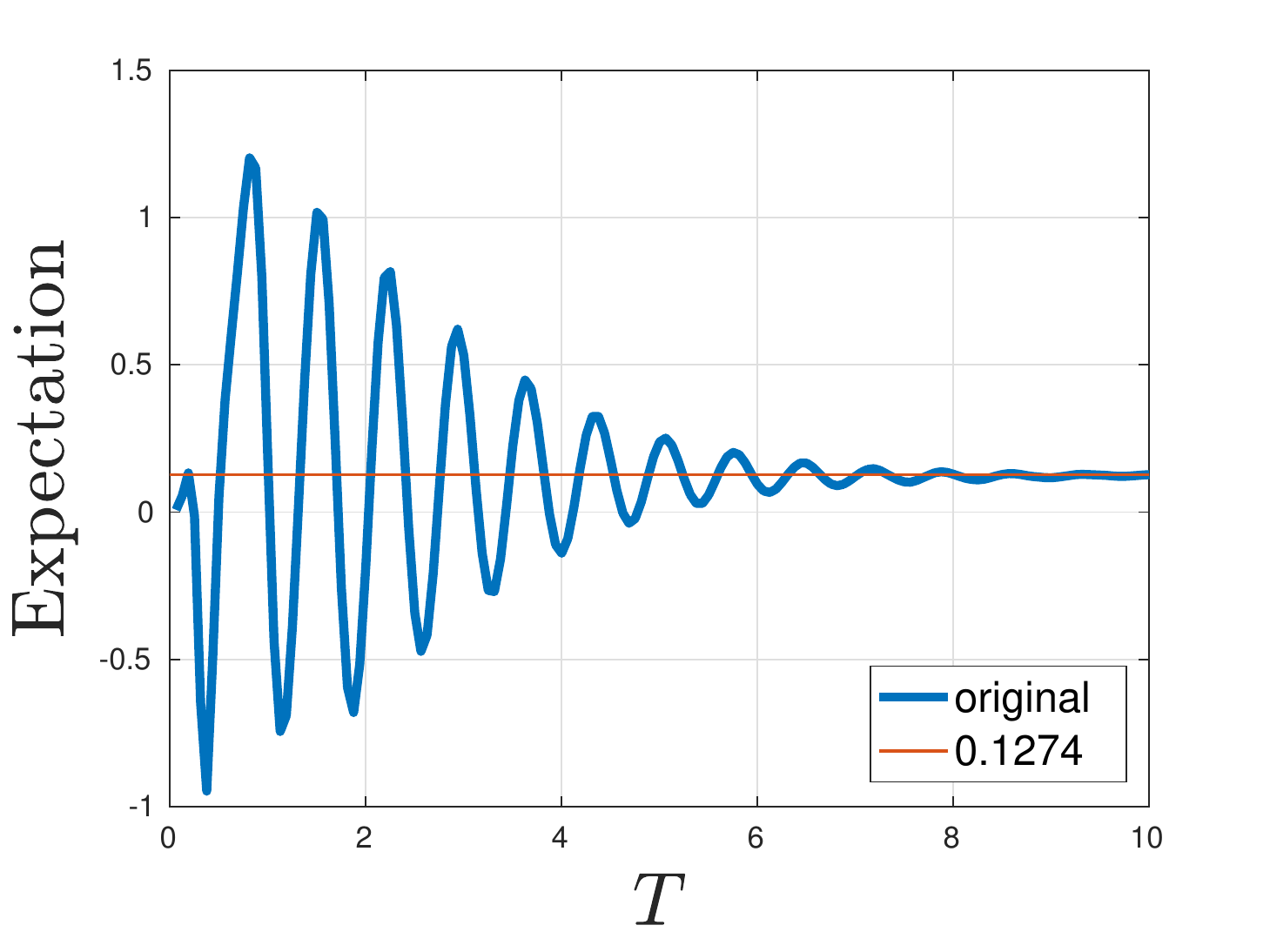}
}
\subfigure[Linear increase of variance]{
\label{Sigma Sen Var}
\includegraphics[width=0.45\textwidth]{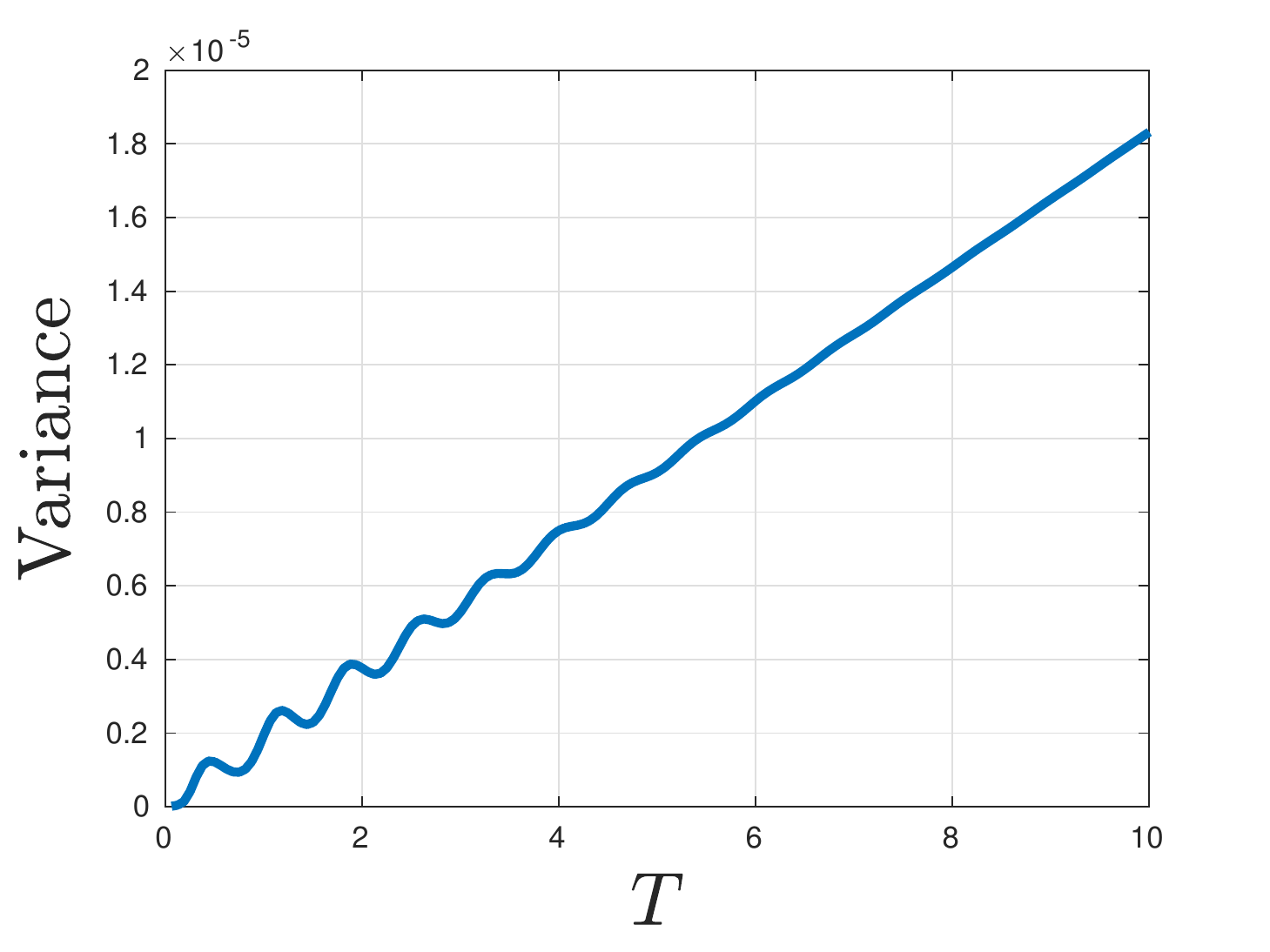}
}
\caption{ The sensitivity with respect to volatility $\sigma$}
\label{Sigma Sensitivity}
\end{figure}

For the initial condition $\xi_0$, we get the new variation process
\[
\D\,z_t =  \left(\frac{\partial f(Y_t) }{\partial Y_t} -SI \right) z_t \, \D t ,
\]
with $z_0=\mathrm{1}$ and the new PS estimator
\begin{equation}
\frac{\partial\breve{F}}{\partial \xi_0} =
\left\langle \nabla \varphi(\hY_T)\,, \hz_T \right\rangle +\varphi(\hY_T)\sum_{n=0}^{N-1} \frac{S}{\sigma} \left\langle \hz_{t_n},\,\Delta W_{n}^{\mathbb{P}} \right\rangle.
\end{equation}

\begin{figure}[h]
\center
\subfigure[Expectation with respect to $T$ ]{
\label{S0 Sen}
\includegraphics[width=0.45\textwidth]{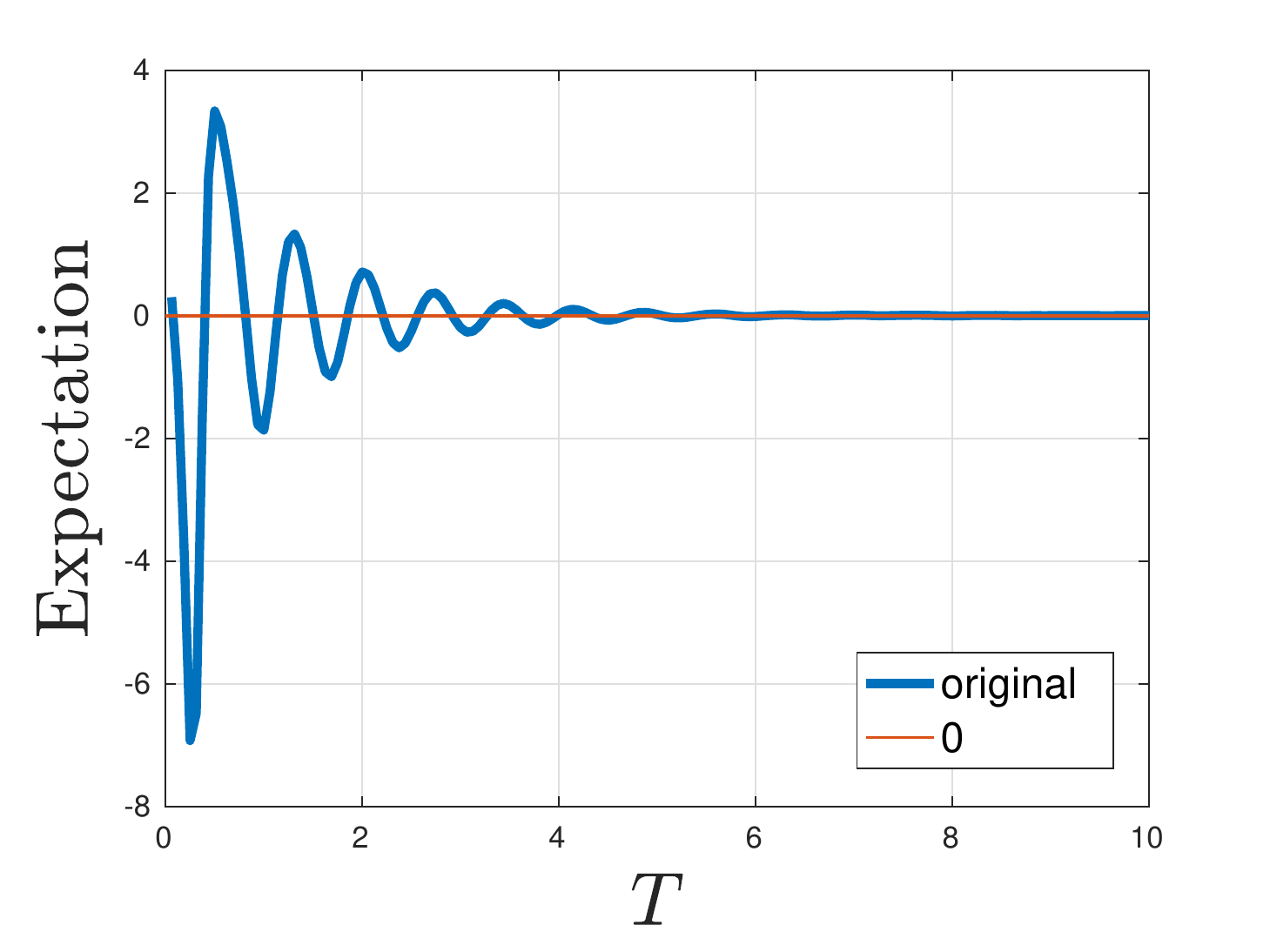}
}
\subfigure[Uniformly bounded variance]{
\label{S0 Sen Var}
\includegraphics[width=0.45\textwidth]{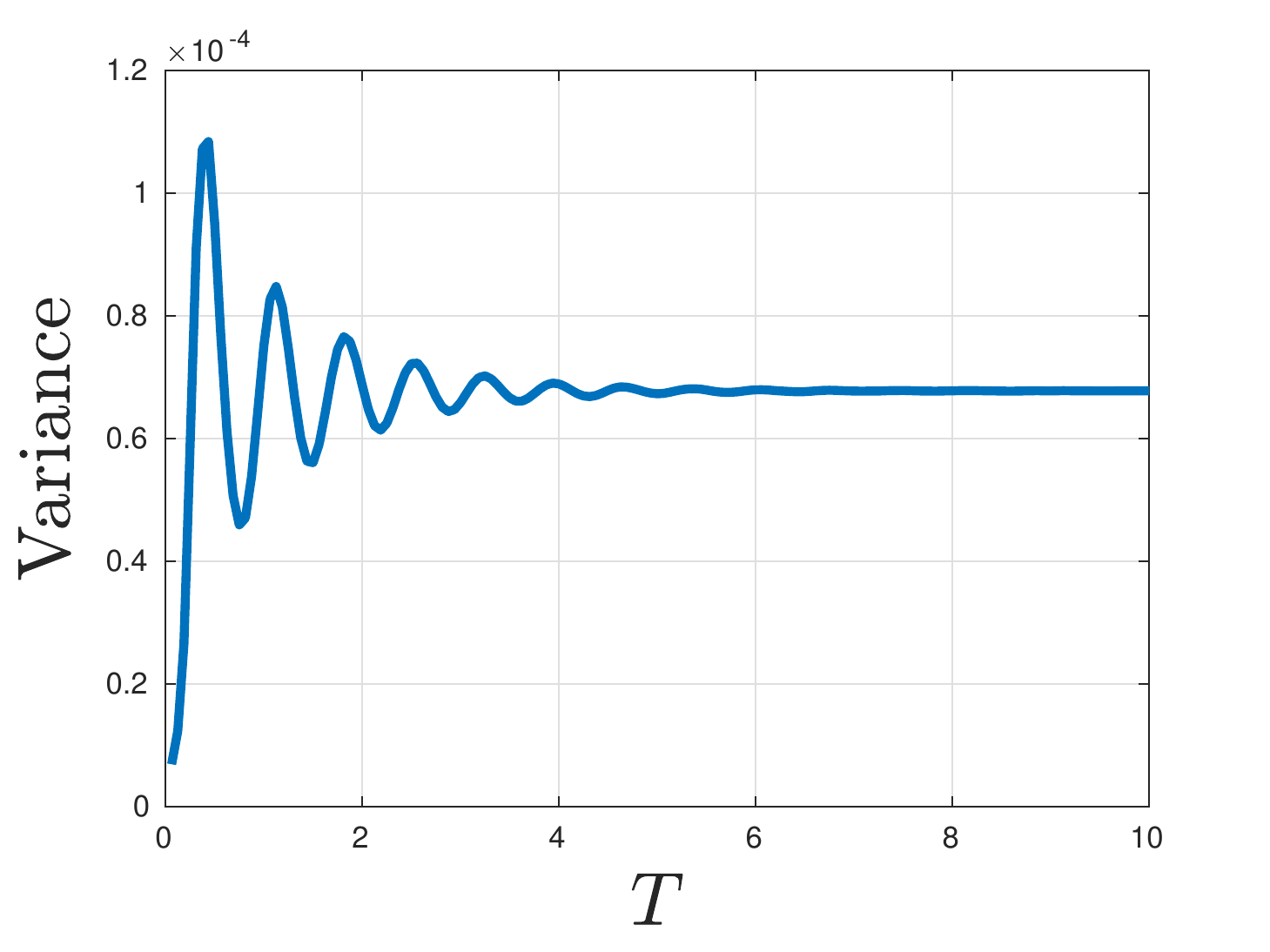}
}
\caption{ The sensitivity with respect to initial value}
\label{S0 Sensitivity}
\end{figure}

Figure \ref{S0 Sen} shows the estimator is ergodic and the expectation converges to $0$ since the initial condition won't affect the invariant measure. Figure \ref{S0 Sen Var} shows the variance is uniformly bounded.

\section{Extension to Lorenz problem} 
In this section, we consider using the new PS estimator for stochastic SDEs to approximate the sensitivities of invariant measures of chaotic ODEs, which is a long-standing problem.
For the invariant measure of a chaotic ODE, the main concept is Sinai-Ruelle-Bowen (SRB) measure. An invariant probability measure $\pi^0$ for a flow $Z_t$ is a physical probability measure if the subset of $z$ satisfying for all continuous function $\varphi$
\[
\lim_{T\rightarrow +\infty} \frac{1}{T} \int_0^T \varphi(Z_t(z))\D t = \pi^0(\varphi).
\]
has positive volume in space. \cite{APPV09} and \cite{TW02} together show that the Lorenz equation has a unique SRB measure. %Theorem C in \cite{AMV15} shows that for Lorenz equation, for any $C^1$ function $\varphi:\ \mathbb{R}^3\rightarrow \mathbb{R},$ there exists a sequence $\xi_0,\xi_1,...$ of i.i.d. normal random variables with mean 0 and finite variance, such that
%\[
%\frac{1}{N} \sum_{n=0}^{N-1} \varphi(Z_n) = \pi^0(\varphi) + \frac{1}{N} \sum_{n=0}^{N-1}\xi_n + O(N^{-3/4}(\log N)^{1/2}(\log\log N)^{1/4}).
%\] 
For the SRB measure of the system, we are also concerned about the the sensitivity with respect to some parameters of the system. It is shown in \cite{RD97,RD09} that the statistical quantities are differentiable with respect to the small perturbations of parameters for the quasi-hyperbolic systems, for example, the Lorenz system.

However, the sensitivity computations of chaotic systems like the Lorenz system are difficult due to the ill-conditioned initial
value problems.  Lea, Allen \& Haine \cite{LAH00} used the Lorenz equation as an example to illustrate the failure of the adjoint methods as the computed sensitivity blows up exponentially. They also proposed an ensemble-adjoint approach, which uses the average of the computed sensitivities in a intermediate time scale. How to choose the time scale is a hard problem and this approach is unlikely to be stable for more complex system. Wang \cite{WQ13} introduced a shadow operator to the original system and found the shadow trajectory by inverting the shadow operator and then calculated the sensitivity. Further, Wang, Hu \& Blonigan \cite{WHB14} found the shadow trajectory by solving a constrained minimization problem and the corresponding convergence result is shown in \cite{WQ14}. 

Our approach is to estimate the sensitivity of the ODEs by estimating the sensitivity of the SDEs with small volatility $\sigma>0.$ The consideration is that, without the random force, the evolution of the invariant measure of ODEs is governed by the corresponding Liouville equation which may not have a well-behaved steady solution. By adding the random noise, the evolution of the invariant measure of SDEs is governed by the corresponding steady Fokker-Planck equation which has a smooth, well-behaved solution. Thuburn \cite{TJ05} appreciated this point and introduced the small diffusion term and calculated the sensitivities by considering the Fokker-Planck equation, but it required the solution of the high-dimensional elliptic PDE numerically which is highly computational expensive. In addition, as shown in Figure \ref{Convergence Time} later this section, the random noise also accelerates the convergence speed to the equilibrium of the system. Next, we first state some theoretical results about the convergence of the solution and invariant measure from SDEs to ODEs, and then show the numerical results for the ordinary Lorenz equation.  

\subsection{Convergence of SDEs to ODEs}
Freidlin \& Wentzell \cite{FW98} provide a systematic framework for the convergence of SDEs to ODEs. Theorem 1.1 in chapter 2 shows that if $f$ is continuous and the ODE has a unique solution $x_t$, then for sufficiently small $\sigma>0,$ we have
\[
\PP\left( \lim_{\sigma\rightarrow 0} \max_{0\leq t\leq T} \left\|X_t^\sigma - x_t\right\|=0\right)=1,
\]
where $X_t^\sigma$ is the solution to the SDE with volatility $\sigma$ and Theorem 1.2 shows that if the drift is Lipschitz and increases no faster than linearly in $x$, then for all $t>0,$ we have
\[
\EE\left[\left\|X_t^\sigma - x_t\right\|^2\right] \leq a(t)\, \sigma^2,
\]
where $a(t)$ is a monotone increasing function depending on the initial data and the Lipschitz constant.

%\textbf{Convergence speed for non-Lipschitz case needs to investigate.}
%\subsection{Invariant measure}
For the invariant measure, Theorem 4.2 in chapter 6 in \cite{FW98} shows that if the drift $f(x)$ is bounded and uniformly continuous, and the deterministic system has a unique invariant measure $\pi^0,$ then $\pi^\sigma$ converges weakly to $\pi^0$ as $\sigma\rightarrow 0,$ where $\pi^\sigma$ is the invariant measure of the SDE with volatility $\sigma.$

%\cite{KP98} provides the convergence results and asymptotic expansion for another chaotic system with limit cycle. It also shows a very different behaviour of stochastic version compared with the original one.

%\textbf{Convergence speed for the invariant measure needs to investigate.}

\subsection{Numerical investigation}
Now we investigate the ordinary Lorenz equation:
\begin{equation}
\begin{aligned}
\D x_t
=\begin{pmatrix}
10(x_t^2-x_t^1)\\
x_t^1(\theta-x_t^3) -x_t^2\\
x_t^1x_t^2-\frac{8}{3}x_t^3
\end{pmatrix}
\D t.
\end{aligned}
\end{equation}
We estimate the sensitivity of $\bar{x}_3 = \lim_{T\rightarrow \infty}\frac{1}{T}\int_0^T x_t^3\ \D t$ with respect to $\theta$ at $\theta=28,$ which is approximately $0.981,$ an average of 50 finite difference estimates using different initial conditions shown in \cite{CNW17}.

Theorem 3.1 in chapter IV in \cite{KY88} shows that the invariant measure $\pi^{\sigma}$ of the stochastic Lorenz equation weakly converges to $\pi^0$ as $\sigma\rightarrow 0$ in discrete version, but the convergence speed is not known. Before we numerically investigate the convergence speed, we first notice that the magnitude of the volatility $\sigma$ affects the speed of convergence starting from the same initial point.

First, we plot the expectation $\EE\left[X_T^{3,\sigma}\right]$ with respect to time $T$ for different $\sigma$ from the same initial point $[-2.4, -3.7, 14.98].$ See Figure \ref{Convergence Time}.

\begin{figure}[h]
\center
\subfigure[Evolutions of different $\sigma$]{
\label{Convergence Time}
\includegraphics[width=0.45\textwidth]{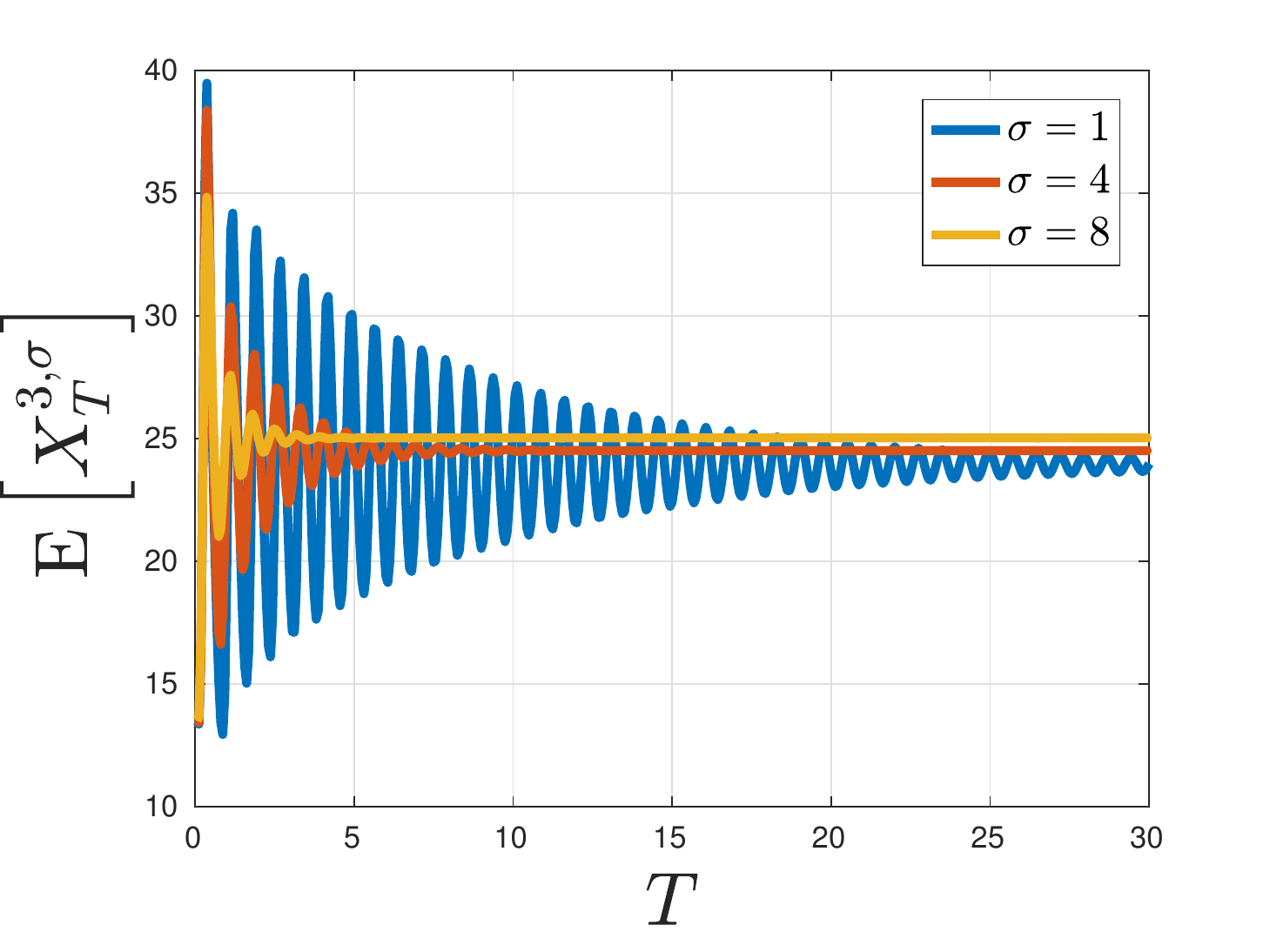}
}
\subfigure[Convergence speed for different $\sigma^2$]{
\label{constant lambda}
\includegraphics[width=0.45\textwidth]{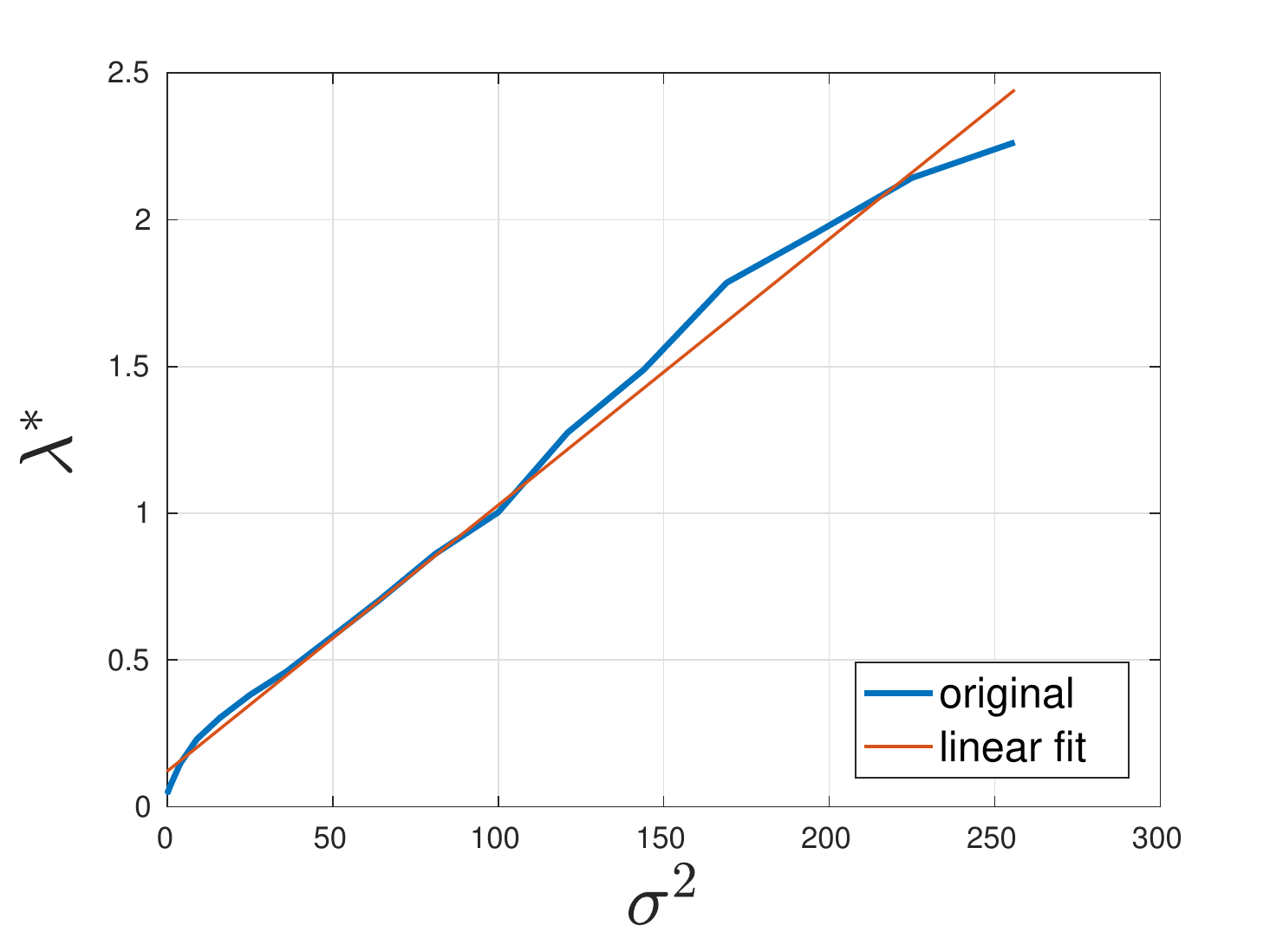}
}
\caption{The effect of $\sigma$ on convergence speed $\lambda^*$}
\label{lambda vs sigma}
\end{figure}

%\begin{figure}[H]
%\center
%\includegraphics[width=0.7\textwidth]{LorenzProblem/%ConvergenceTime}
%\caption{Convergence for different $\sigma$}
%\end{figure}

We see that the time period needed to approach equilibrium decreases and the equilibrium value increases as $\sigma$ increases, which indicates that $\sigma$ affects the convergence speed $\lambda^*.$
Next, to quantify this effect, we calculate the error bound using the moving maximum and minimum and then perform a linear regression of log error bound on $T$ to get the estimated convergence speed $\lambda^*.$ We plot the estimated $\lambda^*$ with respect to $\sigma^2$ in Figure \ref{constant lambda}, which shows clearly that the estimated convergence rate $\lambda^*$ is approximately proportional to $\sigma^2$. Therefore, we need to use a longer $T$ to simulate the SDE with a smaller $\sigma$ to achieve the same truncation error in $T$. It also indicates that the convergence speed of the ODE is far slower than the SDE since ODE has $\sigma=0.$

\begin{figure}[H]
\center
\includegraphics[width=\textwidth]{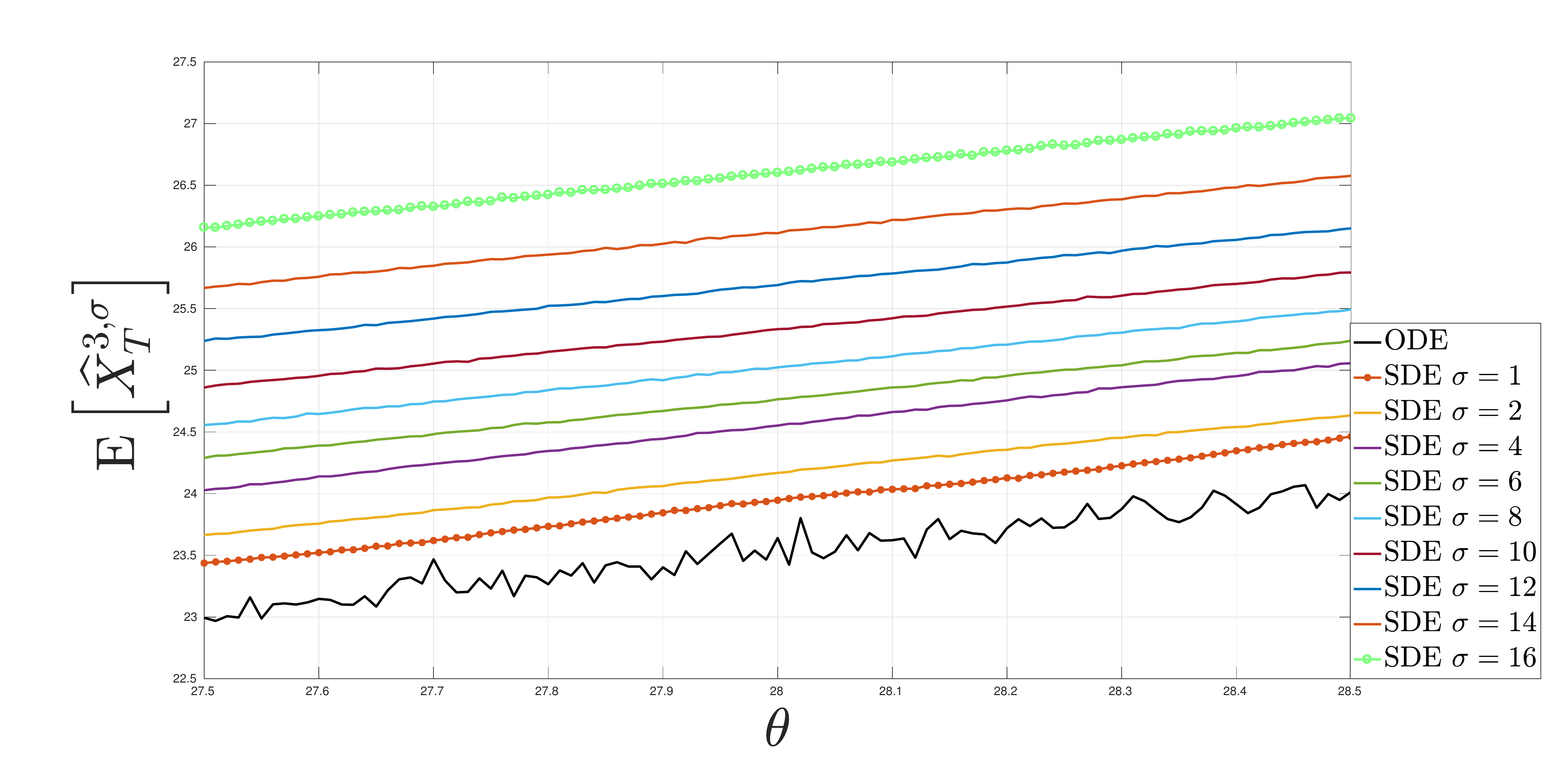}
\caption{ODE and SDEs with different $\sigma$ }
\label{yitong4}
\end{figure}
Next, we investigate the weak convergence order of the invariant measures with respect to $\sigma.$ 
First, we estimate $\bar{x}_3$ for different values of $\theta$ range from $27.50$ to $28.50$ using the $4$th-order Runge-Kutta method with $h=0.001,$ $T=300$ and $x_0 = (-2.4, -3.7, 14.98)$. See the black line in Figure \ref{yitong4}. It oscillates wildly and causes trouble for the finite difference method.

For SDEs, we know from the previous estimates of $\lambda^*$ that we need to choose different times $T$ for different $\sigma$ to ensure that the equilibrium is achieved. Figure \ref{yitong4} shows that, by introducing the small random noise term, the solutions of the SDE are smoother with respect to different $\theta$ and they converge to the solution of the ODE as $\sigma$ decreases.
Simultaneously the sensitivity problem becomes better-posed and the finite difference method already can provide a good estimate of the sensitivity.  

Next, based on Figure \ref{yitong4}, we estimate the weak convergence order with respect to $\sigma,$ see Figure \ref{weak_sigma_3}. It shows first order weak convergence. 
%\begin{figure}[H]
%\center
%\includegraphics[width=0.6\textwidth]{LorenzProblem/Weak_sigma_3}
%\caption{Weak convergence with respect to $\sigma$ }
%\label{weak_sigma_3}
%\end{figure}

%\begin{figure}[H]
%\center
%\includegraphics[width=0.6\textwidth]{LorenzProblem/Weak_sigma_Sen_3}
%\caption{Weak convergence with respect to $\sigma$ }
%\label{weak_sigma_4}
%\end{figure}

\begin{figure}[H]
\center
\subfigure[Original value ]{
\label{weak_sigma_3}
\includegraphics[width=0.45\textwidth]{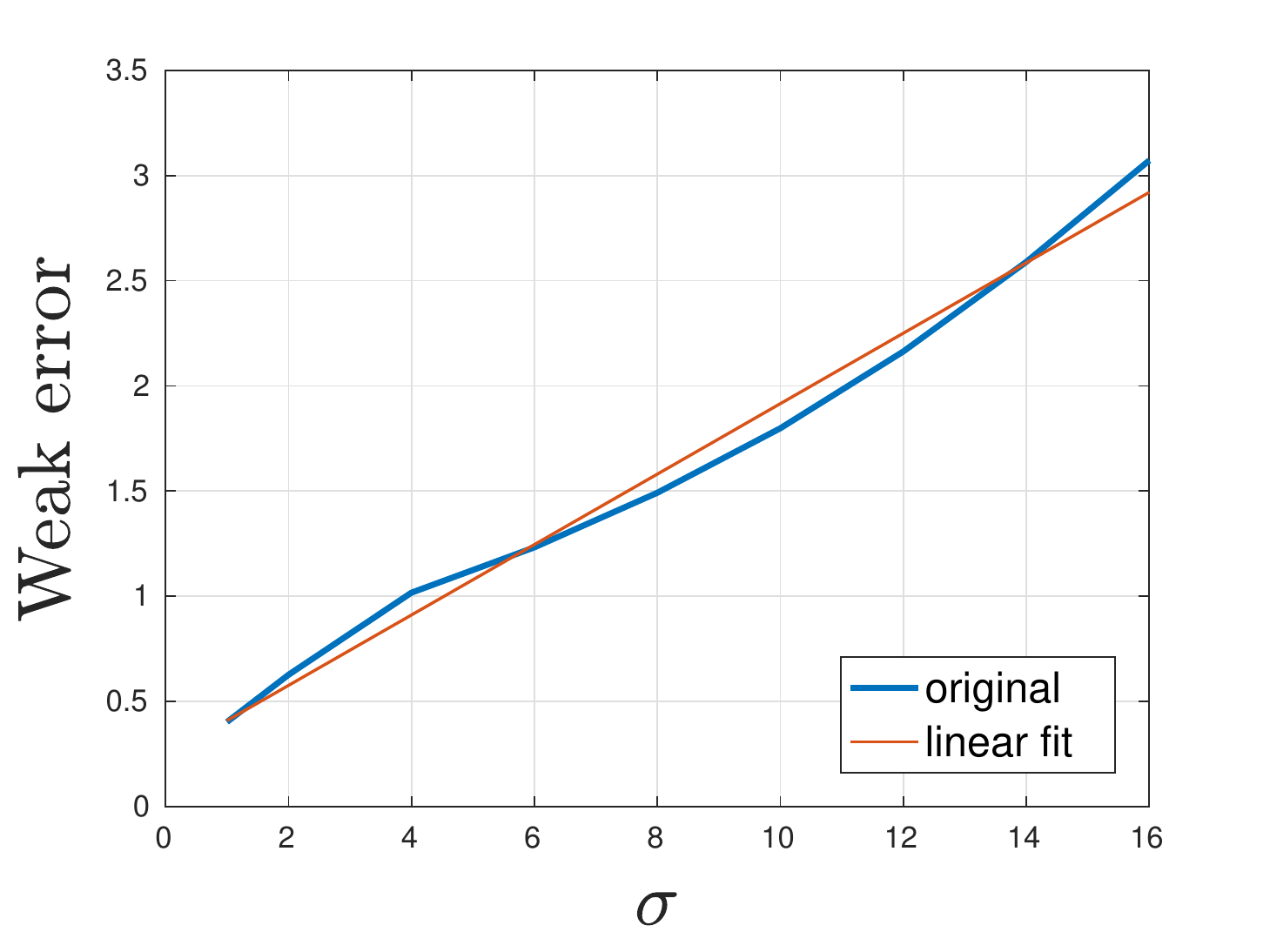}
}
\subfigure[Sensitivity]{
\label{weak_sigma_4}
\includegraphics[width=0.45\textwidth]{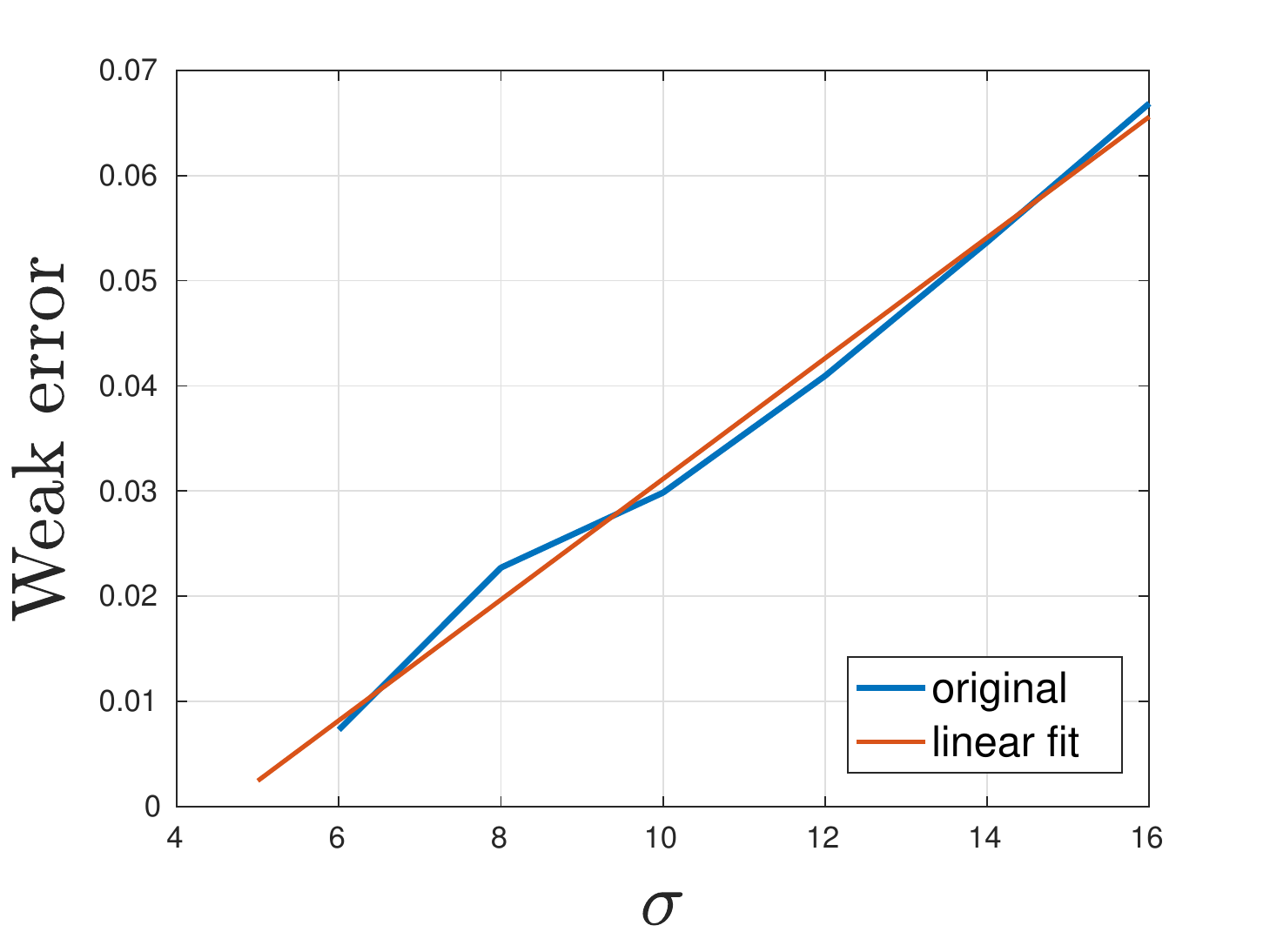}
}
\caption{Weak convergence with respect to $\sigma$}
\end{figure}

Similarly, we plot the evolutions of $\mathrm{E} \left[ \frac{\partial
 \breve{F}}{\partial \theta}(\sigma)\right]$ with respect to $T$ for different $\sigma$ and observe the same results as the original value. Compared with Figure \ref{lambda vs sigma}, the time period needed to achieve the equilibrium of the sensitivities is longer since the estimator is a path integral of the original value.
\begin{figure}[h]
\center
\subfigure[Evolutions of different $\sigma$]{
\label{Convergence Time Sen}
\includegraphics[width=0.45\textwidth]{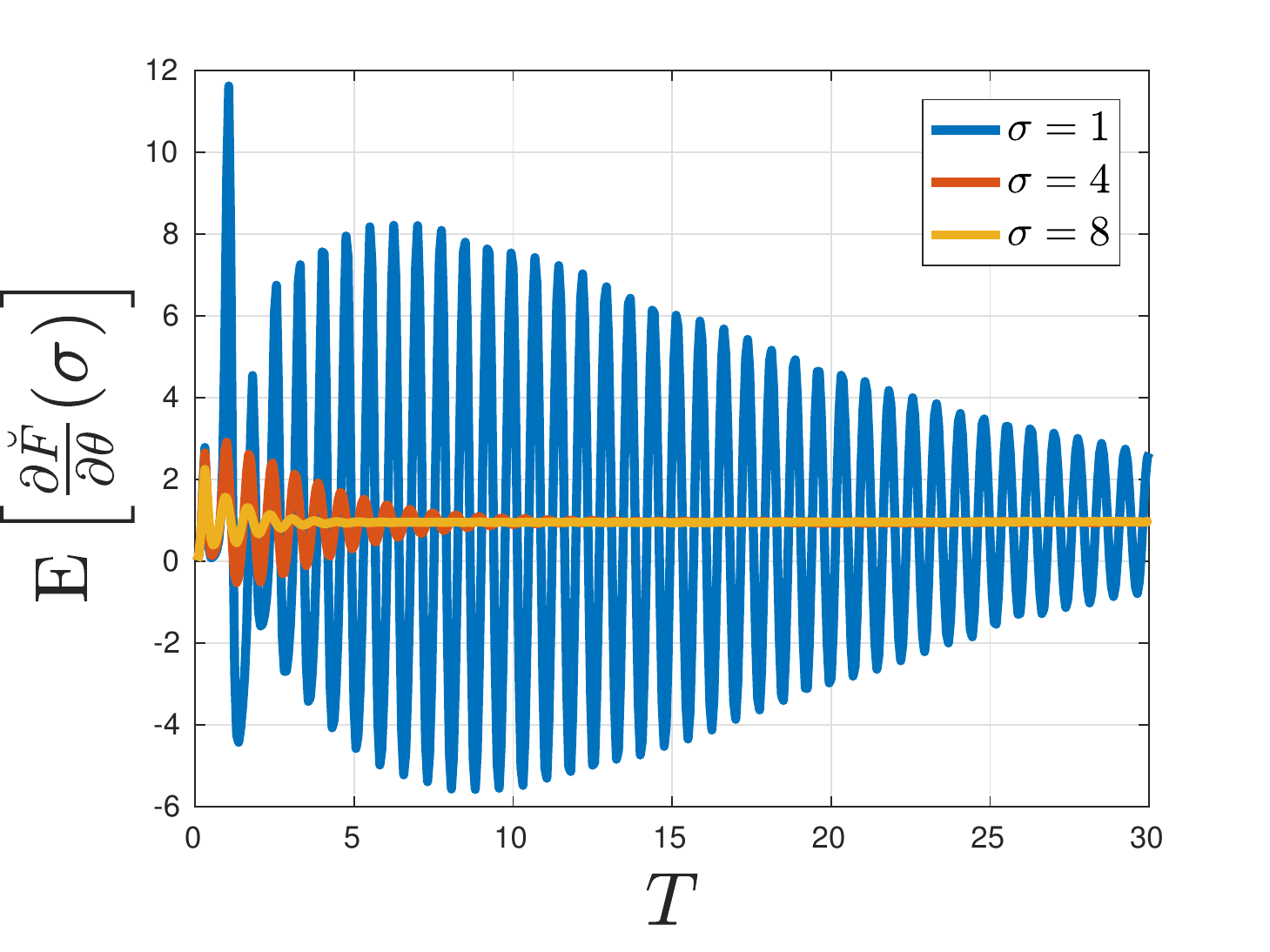}
}
\subfigure[Convergence speed for different $\sigma^2$]{
\label{constant lambda Sen}
\includegraphics[width=0.45\textwidth]{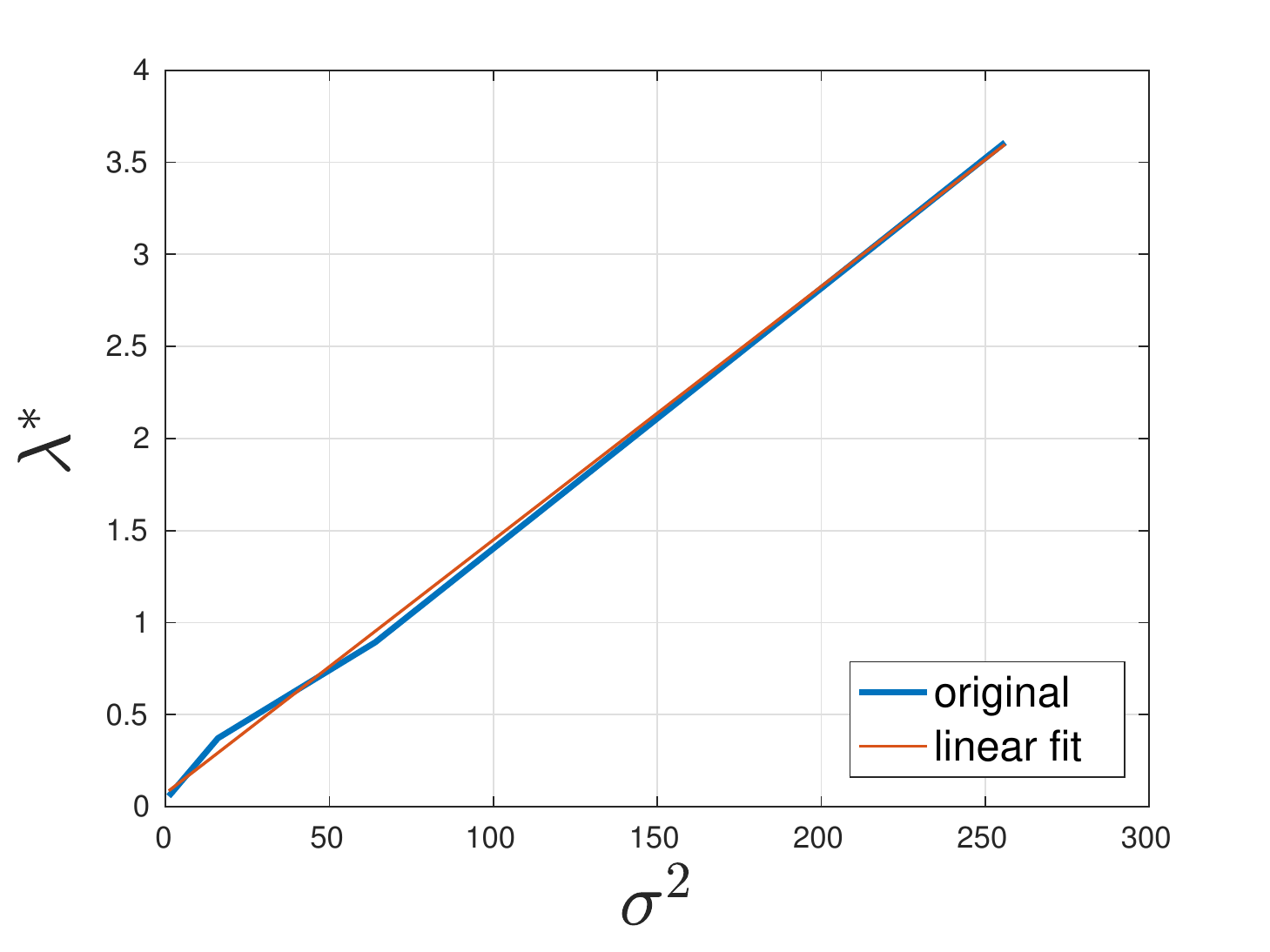}
}
\caption{The effect of $\sigma$ on convergence speed $\lambda^*$}
\label{lambda vs sigma 1 }
\end{figure}

We also plot the sensitivity with respect to $\theta$ for the SDE with different values of $\sigma$ in Figure \ref{yitong6}. The black line shows the sensitivity of the ODE, a constant $0.981$ shown in \cite{CNW17}. For the SDEs, we use the new PS estimator and similarly simulate longer $T$ for smaller $\sigma.$ As shown in Figure \ref{yitong6}, the sensitivities of the SDEs are less smooth with respect to $\theta$ due to the large variance. We plot the average value in the same color and observe that the sensitivity $\mathrm{E} \left[ \frac{\partial
 \breve{F}}{\partial \theta}(\sigma)\right]$ increases as $\sigma$ decreases and converges to the ODE's. The first order weak convergence is shown in Figure \ref{weak_sigma_4}.
\begin{figure}[h]
\center
\includegraphics[width=\textwidth]{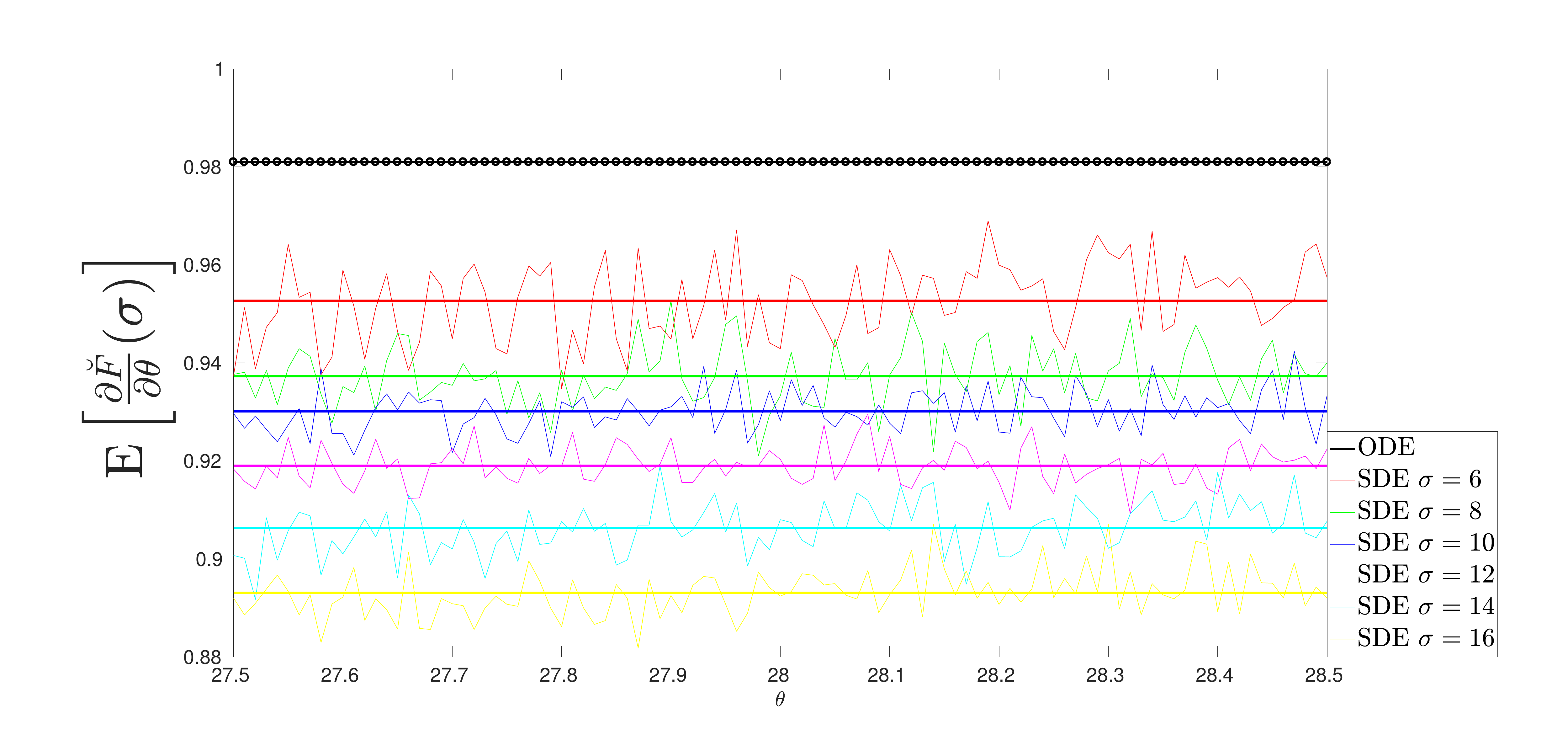}
\caption{Sensitivities of ODE and SDEs with different $\sigma$ }
\label{yitong6}
\end{figure}

%Last, we investigate the effect of $\sigma$ on the variance of the estimators. We use the same $N,h$ and $T$ but with different $\sigma$ to do the numerical experiment. The figure \ref{sigma_variance_sensitivity} shows that the variance of the sensitivity increases linearly in $\sigma^{-2}.$ However, the variance of the original quantity converges to some value which is independent of $\sigma$ as shown in figure \ref{sigma_variance_F}. One possible explanation is that as $\sigma$ decrease to $0$ the invariant measure of the SDE converges to the invariant measure of the ODE, so the variance will converge to the variance of the invariant measure of ODE. For sensitivity, since $Y_t$ and $y_t$ converges to invariant measure but with a factor of $1/\sigma,$ the variance correspondingly increase with respect to $1/\sigma^2.$
%\begin{figure}[H]
%\center
%\subfigure[Effect of $\sigma$ on variance of the sensitivities]{
%\label{sigma_variance_sensitivity}
%\includegraphics[width=0.45\textwidth]{LorenzProblem/Variance_Sigma}
%}
%\subfigure[Effect of $\sigma$ on variance of the original quantity]{
%\label{sigma_variance_F}
%\includegraphics[width=0.45\textwidth]{LorenzProblem/Variance_Sigma_F}
%}
%\caption{Effect of $\sigma$ on variance}
%\end{figure}

Now, we analyze the computational cost first by stating the following numerical results obtained from the previous experiments.
\begin{itemize}
\item The convergence speed $\lambda^*$ is proportional to $\sigma^2$, that is $\lambda^*\sim O(\sigma^2).$  See figure \ref{constant lambda}.
\item The weak convergence order of invariant measure from SDE to ODE with respect to $\sigma$ is 1, see figure \ref{weak_sigma_3} and \ref{weak_sigma_4}.
\item The variance of the new PS estimator linearly increases as $T$ increases see figure \ref{LinearCOMGreek} and as $\sigma^{-2}$ increases by Theorem \ref{Theorem: Importance Sampling}, that is $\mathbb{V}\sim O(T/\sigma^2).$
%\item The variance of the standard MLMC estimator still increases exponentially with respect to $T$ and the change of measure technique need to be employed. Then I suspect the variance without change of measure will increase as $T/\sigma^2 e^{\kappa T}$ and the variance with change of measure will increase as $T^2/\sigma^4,$ where the variance of Radon-Nikodym derivative has an order of  $T/\sigma^2.$  
\end{itemize}

\vspace{1em}

\begin{theorem}[ODE Approximation]
\label{Theorem: ODE}
Suppose the conclusions above hold theoretically, to achieve $O(\varepsilon^2)$ MSE, the optimal complexity of standard Monte Carlo method to approximate the sensitivities of ODEs is $O(\varepsilon^{-9}|\log \varepsilon|^2)$
\end{theorem}
\begin{proof}
Denoting the sensitivity of ODE by $\pi'$ and SDE by $\Pi',$ the weak error can be decomposited into 
\begin{eqnarray*}
\left| \EE\left[  \frac{\partial
 \breve{F}}{\partial \theta}(\sigma)\right]-\pi' \right| & = & \left| \EE\left[ \frac{\partial
 \breve{F}}{\partial \theta}(\sigma)- \frac{\partial
 F}{\partial \theta}(\sigma) + \frac{\partial
 F}{\partial \theta}(\sigma) - \Pi' +\Pi'-\pi' \right]\right|\\
&\leq & \left| \EE\left[ \frac{\partial
 \breve{F}}{\partial \theta}(\sigma)- \frac{\partial
 F}{\partial \theta}(\sigma)\right]\right| +  \left| \EE\left[  \frac{\partial
 F}{\partial \theta}(\sigma) - \Pi'  \right]\right| +  \left| \EE\left[ \Pi'-\pi' \right]\right|\\
& = & O(h) + O(e^{-\lambda^* T}) + O(\sigma).
\end{eqnarray*}
Bounding the weak error by $O(\varepsilon),$ we get
\[
h\sim O(\varepsilon),\ \ \sigma\sim O(\varepsilon),
\] 
and
\[
e^{-\lambda^* T} \sim O(\varepsilon) \Rightarrow \lambda^* T \sim O(|\log \varepsilon|) \Rightarrow T\sim O(\varepsilon^{-2}|\log \varepsilon|),
\]
which implies that the computational cost per path is $T/h \sim O(\varepsilon^{-3}|\log \varepsilon|).$

Next, the variance of MC estimator also needs to be bounded by $O(\varepsilon^2)$ which requires the number of paths $N$ satisfies
\[
T/(\sigma^2 N) \sim O(\varepsilon^2) \Rightarrow N \sim O(\varepsilon^{-6} |\log \varepsilon|).
\]
Therefore, the total computational cost of standard Monte Carlo method is $O(\varepsilon^{-9} |\log \varepsilon|^2).$
\end{proof}
%In \cite{AHS16}, Anderson, Higham \& Sun show that, for the Lipschitz SDEs with small noise, standard MLMC with Euler-Maruyama method fail to improve the computational efficiency when $\varepsilon>\e^{-1/\sigma}$ but can reduce the computational complexity when $\varepsilon\leq \e^{-1/\sigma}.$ Unfortunately, our case is the first kind since $\varepsilon>\e^{-1/\varepsilon}.$

%Milstein \& Tretyakov \cite{MT97} shows under the Lipschitz condition, the MSE of the Lipschitz functional of the numerical solution by Euler-Maruyama method is bounded by $O(h^2 + h\sigma^4)$ which is independent of $\sigma$ when $\sigma$ is sufficiently small. It is consistent with what we observed with the variance of the  estimator of the sensitivity. However, in \cite{AHS16}, they show that the variance of MLMC level estimator is bounded by the order of $O(h^2\sigma^2 + h\sigma^4)$ which may improve the performance of MLMC.

For MLMC, we still require $T\sim O(\varepsilon^{-2}|\log \varepsilon|)$ and the timestep in finest level is $O(\varepsilon).$ The key point here is the efficient coupling, that is the choice of the timestep size  $h_0$ on level 0, which requires $C_0 V_0 > C_1 V_1,$ that is, following Theorem \ref{Theorem: Importance Sampling} and \ref{Theorem: MLMC COM},
\[
 V_1 = O(h_0^2\, T^2/\sigma^4)\sim V_0=O(T/\sigma^2) \Rightarrow h_0 \sim O(\sigma/\sqrt{T}) \Rightarrow h_0 \sim O(\varepsilon^2 |\log \varepsilon|^{-1/2}),
\] 
which means the required $h_0$ is far smaller than the size we need and as a result, MLMC with change of measure doesn't work. The main issue here is the small $\sigma$ destroys the variance of the Radon-Nikodym derivatives. Therefore, constructing an estimator with higher order convergence in $\sigma$ to ODE sensitivity is the key, and we deal with that in the next subsection.

\subsection{Improvement by Richardson-Romberg Extrapolation}
In this subsection, we construct new estimators by applying Richardson-Romberg (R-R) extrapolation on $\sigma$ to improve the weak convergence order of $\sigma$ and then reduce the complexity. R-R extrapolation was first introduced in \cite{TT90} and extended to multi-step R-R extrapolation in \cite{PG07} to achieve higher order weak convergence rate, and further extended to multilevel R-R extrapolation in \cite{LP17}.

We assume that under suitable conditions the sensitivity $\Pi'(\sigma)$ of the invariant measure of SDEs with volatility $\sigma$ has the following expansion at $\sigma=0$ with $\Pi'(0)=\pi'$ which is the sensitivity of the ODEs, 
\begin{equation}
\Pi'(\sigma) = \pi' + \sum_{k=1}^{R-1}c_k\, \sigma^k + O(\sigma^{R})
\end{equation}
where the real constants $c_k,$ for $k=1,...,R-1,$ do not depend on $\sigma.$ Following the multi-step R-R extrapolation, and denoting $\tilde{\Pi}'(\sigma) = (\Pi'(\sigma/1),\Pi'(\sigma/2),...,\Pi'(\sigma/R))^T,$ there exists a weight vector $\tilde{w}=(w_1,...,w_R)\in\mathbb{R}^R$ such that
\begin{equation}
\tilde{w}\,\tilde{\Pi}'(\sigma) = \sum_{k=1}^{R} w_k \,\Pi'(\sigma/k) = \pi' + O(\sigma^R),
\end{equation} 
where $w_k = \frac{(-1)^{R-k}k^R}{k! (R-k)!}$ for $k=1,...,R.$ Specifically, we have
\begin{eqnarray*}
R = 1 &  &w_1 = 1;\\
R = 2 &  &w_1 = -1,\ w_2 = 2; \\
R = 3 &  &w_1 = 1/2,\ w_2=-4,\ w_3 =9/2;%\\
%R = 4 &  &w_1 =-1/6,\ w_2=4,\ w_3=-27/2,\ w_4 = 32/3;
\end{eqnarray*}
Similarly, we construct the new estimators with higher order weak convergence rate on $\sigma.$ For the original expectation $F(\theta,\sigma)$, we have
\begin{eqnarray*}
R = 1 &  &F(\theta,\sigma)\\
R = 2 &  &-F(\theta,2\sigma)+2F(\theta,\sigma); \\
R = 3 &  & \frac{1}{2}\, F(\theta,3\sigma)-4F(\theta,1.5\sigma)+\frac{9}{2}\, F(\theta,\sigma);%\\
\end{eqnarray*}
where we use $\sigma$ as the smallest one for each estimator since the estimator with smallest $\sigma$ converges slowest which determines the simulation time $T.$ The construction for the estimators for sensitivities are exactly the same. Numerically, we plot the weak errors for these three estimators with the same smallest $\sigma = 1/4, 1/2, 1,2,...,20$ for Lorenz problem  in Figure \ref{RR_Exp_weak_error} for original expectation and Figure \ref{RR_Der_weak_error} for sensitivity.
\begin{figure}[h]
\center
\subfigure[Original value]{
\label{RR_Exp_weak_error}
\includegraphics[width=0.45\textwidth]{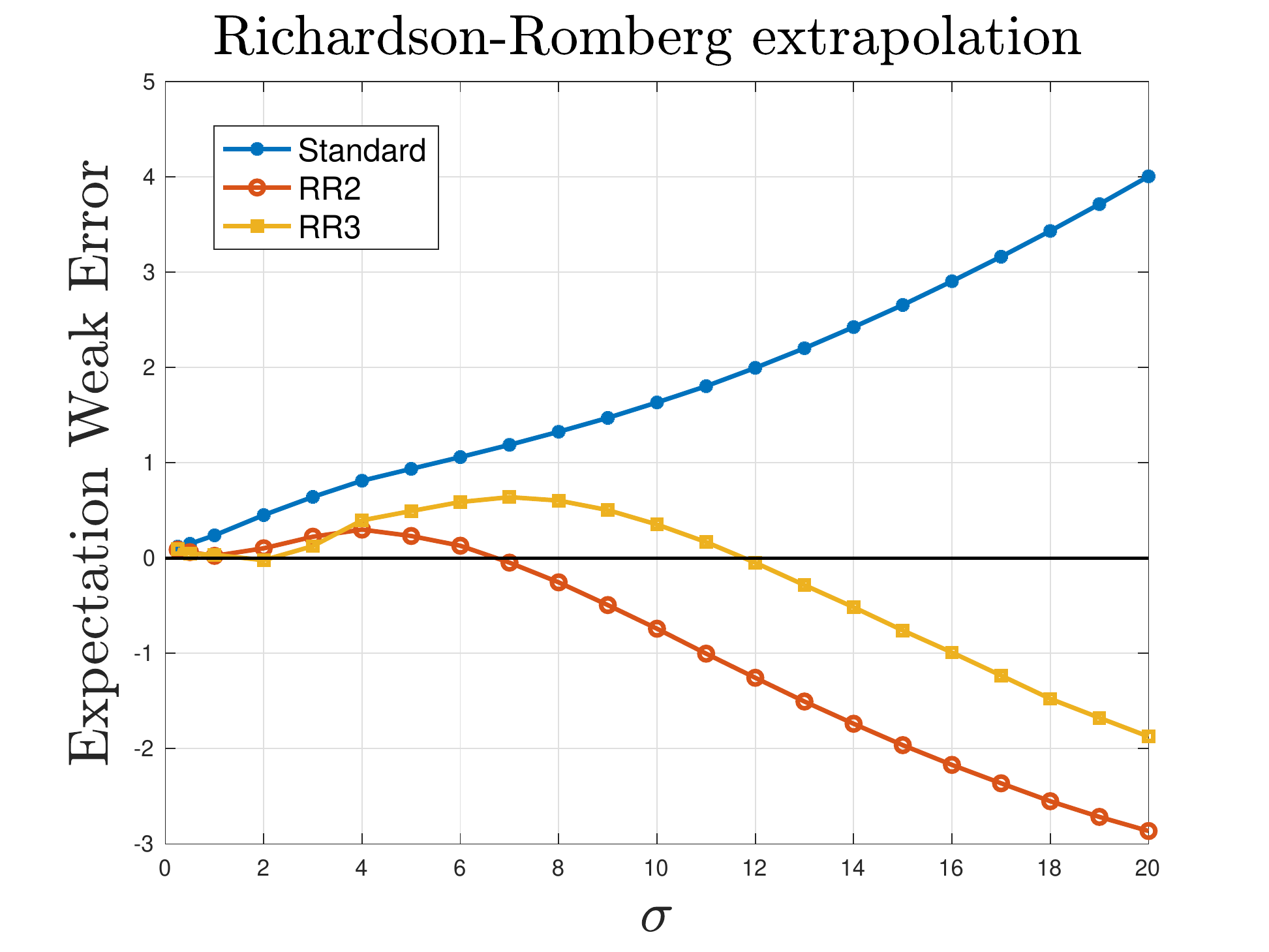}
}
\subfigure[Sensitivity]{
\label{RR_Der_weak_error}
\includegraphics[width=0.45\textwidth]{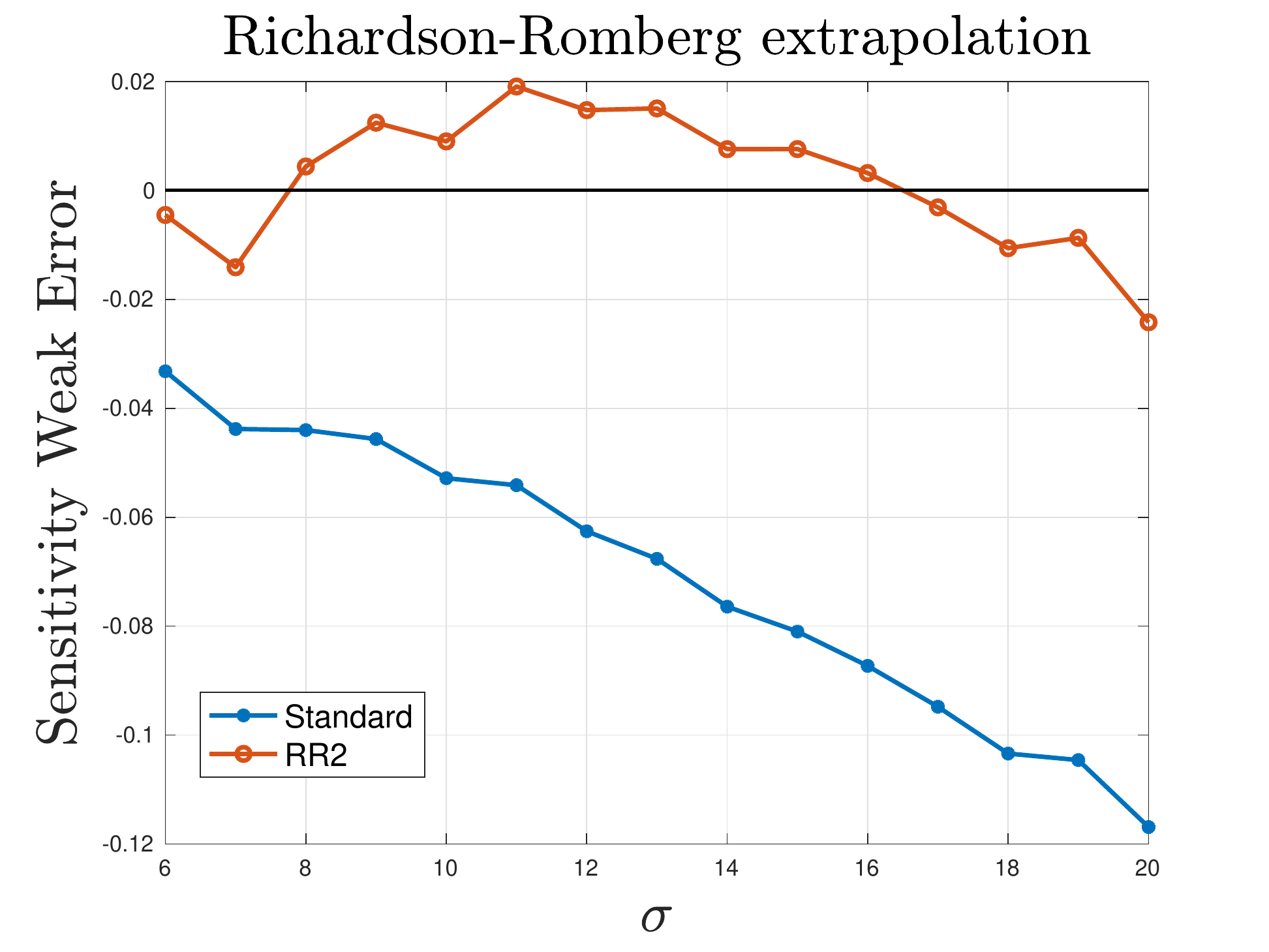}
}
\caption{Weak convergence of different R-R estimators}
\label{RR_summary}
\end{figure}

We see that the R-R extrapolation improve the weak convergence rate significantly. Especially for the sensitivity, the 2nd order R-R estimator already gives an accurate estimate, which may be due to the linearity of the sensitivity on $\sigma.$ The computational cost increases linearly in $R$ since we need to simulate $R$ paths with different $\sigma$ but with the same driving Brownian motion to reduce the variance.

Therefore, following the same approach in Theorem \ref{Theorem: ODE}, for the estimator with order $R$ weak convergence rate on $\sigma,$ to achieve $\varepsilon^2$ MSE, we only need $\sigma\sim O(\varepsilon^{1/R})$ and $T\sim O(\varepsilon^{-2/R}|\log\varepsilon|),$ and reduce the computational cost of standard Monte Carlo method to $O(R\varepsilon^{-3-6/R}|\log\varepsilon|^2).$ The optimal $R$ to minimize the cost is $\ceil{6|\log\varepsilon|}$ which gives a total cost $O(\varepsilon^{-3}|\log\varepsilon|^3).$  
%However, in practice, increasing $R$ to be $\ceil{|\log\varepsilon|/\log|\log\varepsilon|},$ we have the computational cost to be $O(\varepsilon^{-3}|\log\varepsilon|^9),$ which is already optimal up to a logarithmic term. Note that for $\varepsilon=0.02,$ $\ceil{|\log\varepsilon|/\log|\log\varepsilon|}=3.$ 

Similarly, for the MLMC with change of measure, to achieve a good coupling, we require $h_0\sim O(\varepsilon^{2/R}|\log\varepsilon|^{-1/2}),$ and the optimal computational cost is $O(\varepsilon^{-2-8/R}|\log\varepsilon|^3).$ The optimal $R$ to minimize the cost is $\ceil{8|\log\varepsilon|}$ which gives a total cost $O(\varepsilon^{-2}|\log\varepsilon|^4).$ 
%In practice, increasing $R$ to be $\ceil{|\log\varepsilon|/\log|\log\varepsilon|},$ we have the computational cost to be $O(\varepsilon^{-2}|\log\varepsilon|^{12}).$

For the Lorenz problem, Figure \ref{RR_Der_weak_error} shows the 2nd order R-R estimator with $\sigma=15$ is accurate enough. The convergence rate $\lambda^*$ is much larger and we only need to simulate the SDEs to time $T=2$ which is already sufficiently large to achieve equilibrium. Another benefit is that the $h_0$ used on level 0 is much larger. Compared with Figure \ref{MLMC-PS-sum}, the new R-R estimator has a much better accuracy and much smaller computational cost. The following table gives the details of the comparison. Note that the standard estimator still have a large bias due to $\sigma$ and reducing $\sigma$ sufficiently small to achieve a higher accuracy becomes computational infeasible within a reasonable time period, while the 2nd order R-R estimator has achieved a good accuracy with only $1.1\%$ computational cost.

\begin{table}[H]
\centering
\caption{Computation cost comparison }
\label{my-label}
\begin{tabular}{c|c|c|c|c|c}
&$\sigma$ & $T$ & $h_0$ & Computational cost  & Estimated value \\ \hline
Standard &6 & 10 & $2^{-7}$ & $2.450e+13$  & 0.93264\\ \hline
RR2 &15 & 2 & $2^{-4}$ & $2.668e+11$  & 0.97793
\end{tabular}
\end{table}

\section{Conclusions and future work}
In this paper, we proposed a new pathwise sensitivity estimator using importance sampling for the sensitivities of chaotic SDEs. Both the new estimator and the Malliavin estimator perform well and are extended to MLMC successfully. The benefit of the new pathwise estimator is that fundamentally changes the variation process and is easily extended to the sensitivity with respect to initial value and volatility parameters. In addition, we also consider using this estimator for an SDE with small volatility to approximate the sensitivity of an ODE. Together with the Richardson-Romberg extrapolation, we construct a new efficient estimator overcoming the effect of the small volatility.

%One direction for extension of the theory is to develop a more efficient scheme for the approximation of the sensitivity of ODE. We may consider to decrease volatility $\sigma$ as level $l$ increases and try to create a good coupling. Another direction for future research is to consider the expectation of the discontinuous functional $\varphi$ with respect to the invariant measure. For example, the combination of our technique with vibrato Monte Carlo method.

\bibliographystyle{siam}
\bibliography{citation}

\vspace{2em}
\begin{appendix}
%\section{Useful Lemma}

\end{appendix}

\end{document}